\documentclass[a4paper,11pt]{amsart}

\input{xypic}
\allowdisplaybreaks[4]
\newtheorem{proposition}{Proposition}[section]

\newtheorem{theorem}[proposition]{Theorem}

\theoremstyle{definition}
\newtheorem{definition}[proposition]{Definition}
\newtheorem{example}[proposition]{Example}

\newtheorem{blanco}[proposition]{}

\theoremstyle{remark}

\newcommand{\thlabel}[1]{\label{th:#1}}
\newcommand{\thref}[1]{Theorem~\ref{th:#1}}
\newcommand{\selabel}[1]{\label{se:#1}}
\newcommand{\seref}[1]{Section~\ref{se:#1}}

\newcommand{\prlabel}[1]{\label{pr:#1}}
\newcommand{\prref}[1]{Proposition~\ref{pr:#1}}

\newcommand{\bllabel}[1]{\label{bl:#1}}
\newcommand{\blref}[1]{\ref{bl:#1}}

\newcommand{\exlabel}[1]{\label{ex:#1}}
\newcommand{\exref}[1]{Example~\ref{ex:#1}}
\newcommand{\delabel}[1]{\label{de:#1}}
\newcommand{\deref}[1]{Definition~\ref{de:#1}}
\newcommand{\eqlabel}[1]{\label{eq:#1}}
\newcommand{\equref}[1]{(\ref{eq:#1})}

\def\equal#1{\smash{\mathop{=}\limits^{#1}}}

\newcommand{\Hom}{{\rm Hom}}

\def\lan{\langle}
\def\ran{\rangle}
\def\ot{\otimes}

\def\Sets{{\ul{\rm Sets}}}
\def\Fam{{\ul{\rm Fam}}}

\def\GG{{\mathbb G}}

\def\rightactr{\hbox{$\leftharpoonup$}}
\def\leftactr{\hbox{$\rightharpoonup$}}

\newcommand{\Aa}{\mathcal{A}}

\newcommand{\Cc}{\mathcal{C}}
\newcommand{\Dd}{\mathcal{D}}

\newcommand{\Gg}{\mathcal{G}}

\newcommand{\Mm}{\mathcal{M}}

\newcommand{\Tt}{\mathcal{T}}
\newcommand{\Xx}{\mathcal{X}}
\newcommand{\Zz}{\mathcal{Z}}
\newcommand{\YD}{\mathcal{YD}}
\newcommand{\YDT}{\mathcal{YDT}}
\newcommand{\YDZ}{\mathcal{YDZ}}
\def\*C{{}^*\hspace*{-1pt}{\Cc}}
\def\text#1{{\rm {\rm #1}}}

\def\ol{\overline}
\def\ul{\underline}

\begin{document}
\title[Yetter-Drinfeld modules versus Doi-Hopf modules]
{Algebras graded by discrete Doi-Hopf data and the Drinfeld double of 
a Hopf group-coalgebra}
\author{D. Bulacu}
\address{Faculty of Mathematics and Informatics, University
of Bucharest, Strada Academiei 14, RO-010014 Bucharest 1, Romania and
Faculty of Engineering, 
Vrije Universiteit Brussel, B-1050 Brussels, Belgium}
\email{daniel.bulacu@fmi.unibuc.ro}
\thanks{
The first author was supported by FWO-Vlaanderen
(FWO GP.045.09N), and by CNCSIS $479/2009$, code ID 1904.
This research is part of the FWO project G.0117.10 ``Equivariant Brauer groups
and Galois deformations".}
\author{S. Caenepeel}
\address{Faculty of Applied Sciences, 
Vrije Universiteit Brussel, VUB, B-1050 Brussels, Belgium}
\email{scaenepe@vub.ac.be}
\urladdr{http://homepages.vub.ac.be/\~{}scaenepe/}



\begin{abstract}
We study Doi-Hopf data and Doi-Hopf modules for Hopf group-coalgebras.
We introduce modules graded by a discrete Doi-Hopf datum; to a Doi-Hopf datum
over a Hopf group coalgebra, we associate an algebra graded by the underlying
discrete Doi-Hopf datum, using a smash product type construction.
The category of Doi-Hopf modules is then isomorphic to the category of
graded modules over this algebra. This is applied to the category of Yetter-Drinfeld
modules over a Hopf group coalgebra, leading to the construction of the
Drinfeld double. It is shown that this Drinfeld double is a quasitriangular
$\GG$-graded Hopf algebra.
\end{abstract}
\maketitle

\section*{Introduction}
Hopf group coalgebras have been introduced by Turaev \cite{turaev}, and are important
for the study of certain 3-manifolds. A purely algebraic study of Hopf group coalgebras
and related structures was started in \cite{vire}, and continued by several authors,
see for example \cite{wang,wang2,zun}. For a recent survey, we refer to \cite{turaev2}.
A categorical explanation was presented in \cite{cl}, where it was shown that
group coalgebras, resp. Hopf group coalgebras, are coalgebras, resp. Hopf algebras
in a suitable symmetric monoidal category $\Tt_k$. We will recall this construction in
\blref{1.4}; it provides a natural method to generalize results of classical Hopf algebra
theory to the setting of Hopf group coalgebras. For example, it is explained in
\cite[Sec. 4]{cl} how Yetter-Drinfeld modules over Hopf group coalgebras can be
introduced: first we introduce the category of modules over a Hopf group coalgebra,
and then we compute its center, which is a braided monoidal
category, by construction.\\
A crucial result in the classical theory is now the following: to a finite dimensional
Hopf algebra $H$, we can associate a new Hopf algebra $D(H)$, called the Drinfeld
double of $H$. $D(H)$ is quasitriangular, and the category of modules over $D(H)$
is isomorphic to the category of Yetter-Drinfeld modules.
The following natural question now arises: can this result be generalized to the setting
of Hopf group coalgebras? At first glance, this looks like another straightforward application
of the methods developed in \cite{cl}. This is not the case, and the underlying
reason for this is the fact that the category $\Tt_k$ is not rigid. However, we have a
duality functor on $\Tt_k$, but this takes values in a different category $\Zz_k$.\\
To explain this, let us look at a less complicated situation. The dual of a finite dimensional
coalgebra is an algebra, and the category of comodules over the coalgebra is isomorphic
to the category of representations of the dual algebra. If we look at a group coalgebra,
this is a coalgebra in $\Tt_k$, then the dual is an algebra in $\Zz_k$, which turns out to
be a $G$-graded algebra. The category of representations of such an algebra
is then the category of modules graded by a (variable) $G$-set. We have two versions of
this representation category, one with a forgetful functor to $\Tt_k$, and one with a
forgetful functor to $\Zz_k$. There are two corresponding duality results, which are
explained in \blref{1.4bis}.\\
Before looking at Yetter-Drinfeld modules, we consider Doi-Hopf modules; these are more
general, but the formalism is easier, see \cite{cmz}. Doi-Hopf data and Doi-Hopf modules in
$\Tt_k$ are discribed in \blref{1.5} and \blref{1.6}. Our aim is then to describe the category
of Doi-Hopf modules as a category of representations. Before we are able to do this, we have
to introduce a new kind of graded algebra. Recall that a Doi-Hopf datum consists of a
Hopf algebra $H$, an $H$-comodule algebra $A$ and an $H$-module coalgebra $C$.
This construction can be performed in any braided monoidal category, for example in
the category of sets, leading to the notion of discrete Doi-Hopf datum, see \blref{1.2}.
In \seref{2}, we introduce algebras graded by a discrete Doi-Hopf datum $(G,\Lambda,X)$, and in \seref{3},
we discuss modules over a such a graded algebra, graded by a $(G,\Lambda,X)$-set.
Now if we have a Doi-Hopf datum in $\Tt_k$, then the underlying algebra and the dual of
the underlying coalgebra are both algebras, but in different monoidal categories.
However, we can still form their smash product, and this turns out to be an algebra
graded by the underlying Doi-Hopf datum $(G,\Lambda,X)$, see \seref{2}. The main
result of \seref{3} states that the category of Doi-Hopf modules is isomorphic to the
category of modules graded by $(G,\Lambda,X)$-sets, a result that comes in a $\Zz$-version
and in a $\Tt$-version, similar to the duality result in \blref{1.4bis}.\\
In \seref{4}, this result is applied to the category of Yetter-Drinfeld modules: we introduce
the Drinfeld double; it can be constructed as a smash product, and is an algebra graded
by a certain discrete Doi-Hopf datum, denoted $\GG$ throughout the paper.
In Sections \ref{se:5} and \ref{se:6}, we introduce quasitriangular $\GG$-graded Hopf algebras, and
we show that the Drinfeld double is such a quasitriangular $\GG$-graded Hopf algebra.

\section{Preliminaries}\selabel{1}
\begin{blanco}\bllabel{1.1} {\bf Monoidal categories.}
Let $\Cc$ be a monoidal category. One can define algebras, coalgebras, (bi)modules
and (bi)comodules in $\Cc$. If $\Cc$ is braided, then we can consider bialgebras and
Hopf algebras in $\Cc$. Doi-Hopf data and Doi-Hopf modules can then be introduced;
in the case where $\Cc$ is symmetric, this was done in \cite{ph}, and it is easy to see
that this can be extended to arbitrary braided monoidal categories. Let us
briefly recall the definitions. Let $H$ be a bialgebra in $\Cc$. The category $\Cc^H$ of right $H$-comodules
is monoidal, and a right $H$-comodule algebra is an algebra in $\Cc^H$. The category of right
$H$-modules $\Cc_H$ is also monoidal, and a coalgebra in this category is called a right
$H$-module coalgebra. A (right-right) Doi-Hopf datum is a triple $(H,A,C)$, where $H$ is
a bialgebra, $A$ is a right $H$-comodule algebra and $C$ is a right $H$-module coalgebra.
An $(H,A,C)$-Doi-Hopf module is an object $M\in \Cc$ together with a right $A$-action
$\nu_M$ and a right $C$-coaction $\rho_M$ such that thecompatibility condition 
$\rho_M\nu_M=(\nu_M\ot C)(M\ot \psi)(\rho_M\ot A)$, where
 $\psi=(A\ot \nu_C)(c_{C, A}\ot A)(C\ot \rho_A):\ C\ot A\to A\ot C$.
$\psi$ is called the entwining morphism. $c$ is the braiding.
We will denote the category of right $(H, A, C)$-Doi-Hopf 
modules and right $A$-linear right $C$-colinear morphisms
by $\Cc(H)_A^C$.
\end{blanco}

\begin{blanco}\bllabel{1.2} {\bf Discrete Doi-Hopf data.}
Let us describe Doi-Hopf data in $\Sets$. An algebra in $\Sets$ is a monoid. Every set $X$ is in a unique
way a coalgebra in $\Sets$: the comultiplication is the diagonal map $X\to X\times X$, and the
augmentation map is the unique map $X\to \{*\}$, where $\{*\}$ is a singleton. With this coalgebra
structure, every monoid is a bialgebra in $\Sets$. A Hopf algebra in $\Sets$ is then a group.
Let $G$ be a monoid. A $G$-comodule algebra is a monoid $\Lambda$ together with a
morphism of monoids $\gamma:\ \Lambda\to G$. The corresponding $G$-coaction $\Lambda\to
\Lambda\times G$ sends $\lambda$ to $(\lambda,\gamma(\lambda))$. Finally, a $G$-module
coalgebra is a right $G$-set. All these assertions are well-known; they can be proved as easy exercises,
and details can be found in \cite{cl}. We conclude that a Doi-Hopf datum $(G,\Lambda,X)$ in $\Sets$ 
consists of two monoids $G$ and $\Lambda$, a monoid map $\gamma:\ \Lambda\to G$ and
a right $G$-set $X$. We will call $(G,\Lambda,X)$ a discrete Doi-Hopf datum.\\
Now it is easy to show that an object in $\Sets(G)^X_\Lambda$ is a right $\Lambda$-set $Y$
together with a map $\beta:\ Y\to X$ such that $\beta(y\lambda)=\beta(y)\gamma(\lambda)$,
for all $y\in Y$ and $\lambda\in \Lambda$. We call $Y$ a $(G,\Lambda,X)$-set.\\
A morphism $Y\to Y'$ in $\Sets(G)^X_\Lambda$ is a map of right $\Lambda$-sets
$\eta:\ Y\to Y'$ satisfying $\beta'(\eta(y))=\beta(y)$, for all $y\in Y$.\\
An example of a $(G,\Lambda,X)$-set is $Y=\Lambda\times X$, with $\beta(\lambda,x)=x$
and $(\lambda,x)\lambda'=(\lambda\lambda',x\gamma(\lambda'))$.
\end{blanco}

\begin{blanco}\bllabel{1.3} {\bf The Fam-category.}
Let $\Cc$ be a braided monoidal category. To simplify the computations, we assume that $\Cc$ is
strict; this assumptions is justified by the fact that every monoidal category is equivalent to a strict one,
see for example \cite{k}. A new braided monoidal category $\Fam(\Cc)$ is introduced as follows:
objects are families of objects in $\Cc$ indexed by a set $X$, which we denote as
$\ul{M}=(X, (M_x)_{x\in X})$, where $X$ is a set, and $M_x\in \Cc$, for all $x\in X$.
A morphism $\ul{M}\to\ul{M}'$ is a couple $\ul{\varphi}=(f, (\varphi_x)_{x\in X})$, where $f:\ X\to X'$ is
a map and $\varphi_x:\ M_x\to M'_{f(x)}$ is a morphism in $\Cc$. The composition of morphisms
is defined in the obvious way. The tensor product on $\Fam(\Cc)$ is given by
$$\ul{M}\ot \ul{N}=(X\times X', (M_x\ot M'_{x'})_{(x,x')\in X\times X'}).$$
The unit object is $(\{*\},k)$, where $\{*\}$ is a singleton, and $k$ is the unit object of $\Cc$. The
braiding $\ul{c}$ is given by 
$$\ul{c}_{\ul{M},\ul{M}'}=(t, c_{M_x,M'_{x'}}):\ \ul{M}\ot \ul{M}'\to \ul{M}'\ot \ul{M},$$
where $t:\ X\times X'\to X'\times X$ is the switch map, and $c$ is the braiding on $\Cc$.\\
Obviously, we have a strictly monoidal functor $U:\ \Fam(\Cc)\to \Sets$, sending $(X, (M_x)_{x\in X})$
to $X$ and $(f, (\varphi_x)_{x\in X})$ to $f$.
\end{blanco}

\begin{blanco}\bllabel{1.4} {\bf Group-coalgebras.}
Group-coalgebras and Hopf group-coalgebras have been introduced by Turaev in \cite{turaev}.
In \cite{vire}, Virelizier studied Hopf group-coalgebras from an algebraic point of view.
Group-coalgebras and related structures have been investigated by several authors, see
for example \cite{wang,wang2,zun}. For a recent survey, see \cite{turaev2}.\\
Let $k$ be a field (or, more generally, a commutative ring), and $\Mm_k$ the category of
$k$-vector spaces (or, in the case where $k$ is a commutative ring, $k$-modules). Now
consider the categories
$$\Zz_k=\Fam (\Mm_k)~~{\rm and}~~\Tt_k=\Fam(\Mm_k^{\rm op})^{\rm op}.$$
$\Zz_k$ and $\Tt_k$ have the same objects $\ul{M}=(X,(M_x)_{x\in X})$, where $X$ is
a set, and $M_x$ is a $k$-module, for all $x\in X$. For the description of the morphisms in $\Zz_k$,
see \blref{1.3}, with $\Cc$ replaced by $\Mm_k$. A morphism $\ul{M}=(X,(M_x)_{x\in X})
\to \ul{N}=(Y,(N_y)_{y\in Y})$ in $\Tt_k$ is a couple $(f,(\varphi_y)_{y\in Y})$, with $f:\ Y\to X$
a map, and $\varphi_y:\ M_{f(x)}\to N_y$ a $k$-linear map, for every $y\in Y$.\\
It was observed in
\cite{cl} that a group-coalgebra (resp. a Hopf group-coalgebra) is a coalgebra (resp. a Hopf algebra)
in $\Tt_k=\Fam(\Mm_k^{\rm op})^{\rm op}$. An algebra in $\Tt_k$ is a collection of $k$-algebras
indexed by a set $X$.\\
In a similar way, a coalgebra in $\Zz_k$ is a collection of $k$-coalgebras indexed by a set $X$.
Algebras in $\Zz_k$ are in one-to-one correspondence to algebras graded by a monoid.
\end{blanco}

\begin{blanco}\bllabel{1.4bis} {\bf Isomorphism of categories.}
Let $\Cc$ be as in \blref{1.3}, and consider the subcategory $\Fam^{bij}(\Cc)$ of $\Fam(\Cc)$,
with the same objects as $\Fam(\Cc)$, but with morphisms of the form $(f,(\varphi_x)_{x\in X})$,
with $f$ a bijection. Then we have an isomorphism $F$ between the categories
$\Zz_k^{bij}$ and $\Tt_k^{bij}$, acting as the identity on objects. At the level of morphisms,
$F$ is defined by
$$F(f, (\varphi_x)_{x\in X})=(f^{-1}, (\varphi_{f^{-1}(y)})_{y\in Y}).$$
\end{blanco}

\begin{blanco}\bllabel{1.4ter} {\bf A duality result.}
It is well-known that the dual $B=C^*$ of a $k$-coalgebra $C$ is a $k$-algebra; we have a
functor $\Mm^C\to \Mm_{B^{\rm op}}$, which is an isomorphism of categories if $C$ is finitely generated
and projective as a $k$-module.\\
We will now discuss a similar result for group coalgebras. Actually, it is a special case of a more
general duality result that will be discussed in the subsequent sections. But this special case
might be illuminating, as it incorporates some of the subtleties that will reappear later in a more
general situation, and this is why we decided to give an outline here.\\
Let $A$ be a $G$-graded $k$-algebra, and $\Zz_A$ the category of right modules over $A$, viewed
as an algebra in $\Zz_k$. The objects of $\Zz_A$ are couples $(X,M)$, where $X$ is a right $G$-set,
and $M=\oplus_{x\in X} M_x$ is a right $A$-module graded by $X$; we refer to \cite{NastasescuRV90}
for detail on modules graded by $G$-sets. A morphism $(X,M)\to (Y,N)$ in $\Zz_A$ is a couple
$(f,\varphi)$, where $f:\ X\to Y$ is a morphism of $G$-sets, and $\varphi:\ M\to N$ is a right
$A$-module map such that $\varphi(M_x)\subset N_{f(x)}$, for all $x\in X$. It is obvious that we
have a forgetful functor $\Zz_A\to \Zz_k$.\\
We have a second category $\Tt_A$, with the same objects as $\Zz_A$, but morphisms defined in
a different way: 
$(f,(\varphi_y)_{y\in Y}):\ (X,M)\to (Y,N)$
consists of a morphisms of right $G$-sets $f:\ Y\to X$, and a bunch of $k$-linear maps
$\varphi_y:\ M_{f(y)}\to N_y$ such that
$\varphi_{yg}(ma)=\varphi_y(m)a$,
for all $m\in M_{f(y)}$ and $a\in A_g$. Observe that we have a forgetful functor $\Tt_A\to \Tt_k$.\\
In a similar way, we have two categories associated to a group coalgebra $\ul{C}=
(G,(C_g)_{g\in G})$, one with a forgetful functor to $\Tt_k$, and the other one with a forgetful functor
to $\Zz_k$. The two categories $\Tt^{\ul{C}}$ and $\Zz^{\ul{C}}$ have the same objects
$(X,(M_x)_{x\in X})$, where $X$ is a right $G$-set, $M_x$ is a $k$-module, and
$$\rho_{x,g}:\ M_{xg}\to M_x\ot C_g$$
are $k$-linear maps such that the following coassociativity and counit conditions hold:
$$(M_x\ot \Delta_{g,h})\circ \rho_{x,gh}=(\rho_{x,g}\ot C_h)\circ \rho_{xg,h}~~;~~
(M_x\ot \varepsilon)\circ \rho_{x,e}=M_x.$$
A morphism $\ul{M}\to \ul{N}$ in $\Tt^{\ul{C}}$ is a morphism $(f,(\varphi_y)_{y\in Y})$ in $\Tt_k$
such that $f$ is a morphism of $G$-sets and
$$(\varphi_y\ot C_g)\circ \rho_{f(y),g}=\rho_{y,g}\circ \varphi_{yg}.$$
A morphism $\ul{M}\to \ul{N}$ in $\Zz^{\ul{C}}$ is a morphism $(f,(\varphi_x)_{x\in X})$ in $\Zz_k$
such that $f$ is a morphism of $G$-sets and
$$(\varphi_x\ot C_g)\circ \rho_{x,g}=\rho_{f(x),g}\circ \varphi_{xg}.$$
Now let $\ul{C}$ be a group coalgebra, and suppose that the underlying monoid $G$ is a group.
Write $B_g=C^*_{g^{-1}}$. Then $B=\oplus_{g\in G} B_g$ is a $G$-graded $k$-algebra, with
multiplication maps $B_g\ot B_{h}\to B_{gh}$ given by opposite convolution: for
$\xi\in C^*_{g^{-1}}$, $\xi'\in C^*_{h^{-1}}$ and $c\in C_{(gh)^{-1}}$, we have
$$(\xi\xi')(c)=\xi(c_{(2,g^{-1})})\xi'(c_{(1,h^{-1})}).$$
We have functors
$T:\ \Tt^{\ul{C}}\to \Tt_B$ and $Z:\ \Zz^{\ul{C}}\to \Zz_B$
defined as follows: at the level of objects, $T$ and $F$ are defined in the same way:
$$T(X,(M_x)_{x\in X})=Z(X,(M_x)_{x\in X})=\bigoplus_{x\in X} M_x,$$
with the following right $B$-action: for $m\in M_x$ and $\xi\in B_g=C_{g^{-1}}^*$:
$$m\xi=\lan \xi,m_{[1,g^{-1}]}\ran m_{[0,xg]}.$$
At the level of morphisms, $T$ and $Z$ are the identities. If every $C_g$ is finitely generated
and projective as $k$-modules, then $T$ and $F$ are isomorphisms of categories.
The inverse functors are defined as follows: if $M$ is graded by the $G$-set $X$, then
$T^{-1}(M)=Z^{-1}(M)= (X,(M_x)_{x\in X})$, with coaction maps
$\rho_{x,g}:\ M_{x,g}\to M_x\ot C_g$ given by the formula
$$\rho_{x,g}(m)=m\xi^{(g)}\ot c^{(g)}.$$
We implicitly introduced the following notation, which will be used throughout the rest of this
paper. It is well known that $C_{g}$ is finitely generated
and projective if and only if there exists a unique $\xi^{(g)}\ot c^{(g)}\in C_{g}^*
\ot C_{g}$, called finite dual basis of $C_g$ (summation is implicitly understood) such that $c=\xi^{(g)}(c) c^{(g)}$
and $\xi=\xi( c^{(g)})\xi^{(g)}$ for all $c\in C_g$ and $\xi\in C_g^*$.
Also observe that
\begin{equation}\eqlabel{2.9.2}
h\leftactr \xi^{(g)}\ot c^{(g)}=\xi^{(g)}\ot c^{(g)}h,
\end{equation}
for all $h\in H_{\gamma(g)}$. Indeed, let 
$\xi^{(g)}\ot c^{(g)}=\tilde{\xi}^{(g)}\ot \tilde{c}^{(g)}$. Then
$$
h\leftactr \xi^{(g)}\ot c^{(g)}=(h\leftactr \xi^{(g)})(\tilde{c}^{g})
\tilde{\xi}^{g}\ot c^{(g)}
=
\tilde{\xi}^{g}\ot (h\leftactr \xi^{(g)})(\tilde{c}^{g})c^{(g)}=
\xi^{(g)}\ot c^{(g)}h.$$
\end{blanco}

\begin{blanco}\bllabel{1.5} {\bf Doi-Hopf data in $\Tt_k$.}
First, let $\ul{H}=
(G, (H_g)_{g\in G})$ be a semi-Hopf group coalgebra, that is a bialgebra in $\Tt_k$. This means that
we have the following data and properties:
\begin{itemize}
\item $G$ is a monoid;
\item $H_g$ is a $k$-algebra, for every $g\in H$;
\item we have $k$-algebra maps $\varepsilon:\ H_e\to k$ and $\Delta_{g,g'}:\
H_{gg'}\to H_g\ot H_{g'}$, for all $g,g'\in G$.
\end{itemize}
The Sweedler notation for the comultiplication maps is the following:
$$\Delta_{g,g'}(h)=h_{(1,g)}\ot h_{(2,g')}.$$
The following coassociativity and counit property have to be satisfied:
$$(H_g\ot \Delta_{g',g''})\circ \Delta_{g,g'g''}=(\Delta_{g,g'}\ot H_{g''})\circ \Delta_{gg',g''};$$
$$(H_g\ot \varepsilon)\circ \Delta_{g,e}=(\varepsilon\ot H_g)\circ\Delta_{e,g}= H_g.$$
Now let $\ul{A}=(X,(A_x)_{x\in X})$ be a right $\ul{H}$-comodule algebra. This means that we have
the following data and properties:
\begin{itemize}
\item $X$ is a right $G$-set;
\item $A_x$ is a $k$-algebra, for all $x\in X$;
\item we have $k$-algebra maps $\rho_{x,g}:\ A_{xg}\to A_x\ot H_g$; 
\end{itemize}
The following coassociativity and counit properties have to hold:
$$(A_x\ot \Delta_{g,h})\circ \rho_{x,gh}=(\rho_{x,g}\ot H_h)\circ \rho_{xg,h};$$
$$(A_x\ot\varepsilon)\circ \rho_{x,e}=A_x.$$
We use the following Sweedler-type notation for the coaction maps:
$$\rho_{x,g}(a)=a_{[0,x]}\ot a_{[1,g]}.$$
Finally, let $\ul{C}=(\Lambda, (C_\lambda)_{\lambda\in \Lambda})$ be a right $\ul{H}$-module
coalgebra. This means that we have the following:
\begin{itemize}
\item $\Lambda$ is a monoid, and we have a monoid morphism $\gamma:\ \Lambda\to G$;
\item $C_\lambda$ is a right $H_{\gamma(\lambda)}$-module, for every $\lambda\in \Lambda$;
\item $\ul{C}$ is a group-coalgebra, that is, we have $k$-linear maps
$\Delta_{\lambda,\lambda'}:\ C_{\lambda\lambda'}\to C_{\lambda}\ot C_{\lambda'}$
and $\varepsilon:\ C_e\to k$ satisfying the appropriate coassociativity and counit properties;
\item 
the following compatibility conditions have to be fulfilled: for all $c\in C_{\lambda\lambda'}$ and
$h\in H_{\gamma(\lambda\lambda')}$, we have
$$\Delta_{\lambda,\lambda'}(ch)=c_{(1,\lambda)}h_{(1,\gamma(\lambda))}
\ot c_{(2,\lambda')}h_{(2,\gamma(\lambda'))},$$
and $\varepsilon(ch)=\varepsilon(c)\varepsilon(h)$, for all $c\in C_e$, $h\in H_e$.
\end{itemize}
Observe that $(G,\Lambda,X)$ is a discrete Doi-Hopf datum; we call it the discrete Doi-Hopf
datum underlying $(\ul{H},\ul{A},\ul{C})$.
\end{blanco}

\begin{blanco}\bllabel{1.6} {\bf Doi-Hopf modules in $\Tt_k$.}
Now we describe the objects of $\ul{M}=(Y,(M_y)_{y\in Y})\in\Tt_k(\ul{H})_{\ul{A}}^{\ul{C}}$. These consist of the following
data
\begin{itemize}
\item a $(G,\Lambda,X)$-set $Y$ (see \blref{1.2});
\item for every $y\in Y$, a right $A_{\beta(y)}$-module $M_y$;
\item $k$-linear maps $\rho_{y,\lambda}:\ M_{y\lambda}\to M_y\ot C_\lambda$, satisfying the appropriate
coassociativity and counit conditions;
\item 
the following compatibility conditions have to be satisfied, for all $m\in M_{y\lambda}$ and
$a\in A_{\beta(y\lambda)}$:
\end{itemize}
\begin{equation}\eqlabel{1.4.1}
\rho_{y,\lambda}(ma)=m_{[0,y]}a_{[0,\beta(y)]}\ot m_{[1,\lambda]}a_{[1,\gamma(\lambda)]}.
\end{equation}
Here we use the following Sweedler-type notation for the coaction on $\ul{M}$:
$\rho_{y,\lambda}(m)=m_{[0,y]}\ot m_{[1,\lambda]}$, for $m\in M_{y\lambda}$.\\
A morphism $\ul{M}=(Y,(M_y)_{y\in Y})\to \ul{M}'=(Y',(M'_{y'})_{y'\in Y'})$ in $\Tt_k(\ul{H})_{\ul{A}}^{\ul{C}}$
is a couple $(\eta, (\varphi_{y'})_{y'\in Y'})$, where
\begin{itemize}
\item $\eta:\ Y'\to Y$ is a morphism of $(G,\Lambda,X)$-sets;
\item for every $y'\in Y'$, $\varphi_{y'}:\ M_{\eta(y')}\to M'_{y'}$ is a right $A_{\beta'(y')}$-linear map;
\item for all $y'\in Y'$ and $\lambda\in \Lambda$,  diagram \equref{1.4.2} commutes.
\end{itemize}
\begin{equation}\eqlabel{1.4.2}
\xymatrix{
M_{\eta(y')\lambda}\ar[d]_{\rho_{\eta(y'),\lambda}}\ar[rr]^{\varphi_{y'\lambda}}&&
M'_{y'\lambda}\ar[d]^{\rho'_{y',\lambda}}\\
M_{\eta(y')}\ot C_\lambda\ar[rr]^{\varphi_{y'}\ot C_\lambda}&& M'_{y'}\ot C_\lambda}
\end{equation}
\end{blanco}

\begin{example}\exlabel{1.6}
Let $G$ be a monoid; then $(G,G,G)$ is a discrete Doi-Hopf datum. $G$ is a right
$G$-module by right multiplication, and the identity on $G$ is a morphism of monoids.
A $(G,G,G)$-set is a right $G$-set $Y$ together with a map $\beta:\ Y\to G$
satisfying $\beta(yg)=\beta(y)g$, for all $y\in Y$ and $g\in G$.\\
Let $\ul{H}=(G, (H_g)_{g\in G})$ be a semi-Hopf group coalgebra. $(\ul{H},\ul{H},\ul{H})$
is a Doi-Hopf datum in $\Tt_k$. A Doi-Hopf module $(Y, (M_y)_{y\in Y})\in
\Tt_k(\ul{H})^{\ul{H}}_{\ul{H}}$ consists of the following data: $Y$ is $(G,G,G)$-set
as above; every $M_y$ is a right $H_{\beta(y)}$-module, and
$\rho_{yg}:\ M_{yg}\to M_y\ot H_g$ is a coassociative coaction. For every
$m\in M_{yg}$ and $h\in H_{\beta(yg)}$, we have the compatibility relation
$$\rho_{y,g}(mh)=m_{[0,y]}h_{(1,\beta(y))}\ot m_{[1,g]}h_{(2,g)}.$$
These Doi-Hopf modules are simply called Hopf modules, and have been considered
in \cite[Sec. 3.1]{cl}, where the Structure Theorem for Doi-Hopf modules was discussed.
\end{example}

\begin{blanco}\bllabel{1.7}
Let $\Zz_k(\ul{H})_{\ul{A}}^{\ul{C}}$ be the category with the same objects as
$\Tt_k(\ul{H})_{\ul{A}}^{\ul{C}}$, but with morphisms defined in a different way.
A morphism $\ul{M}=(Y,(M_y)_{y\in Y})\to \ul{M}'=(Y',(M'_{y'})_{y'\in Y'})$ in $\Zz_k(\ul{H})_{\ul{A}}^{\ul{C}}$
is a couple $(\eta, (\varphi_{y})_{y\in Y})$, where
\begin{itemize}
\item $\eta:\ Y\to Y'$ is a morphism of $(G,\Lambda,X)$-sets;
\item for every $y\in Y$, $\varphi_y:\ M_y\to M'_{\eta(y)}$ is a morphism of
$A_{\beta(y)}=A_{\beta'(\eta(y))}$-modules;
\item for all $y\in Y$ and $\lambda\in \Lambda$, diagram \equref{1.7.1} commutes.
\end{itemize}
\begin{equation}\eqlabel{1.7.1}
\xymatrix{
M_{y\lambda}\ar[d]_{\rho_{y,\lambda}}\ar[rr]^{\varphi_{y\lambda}}&&
M'_{\eta(y)\lambda}\ar[d]^{\rho'_{\eta(y),\lambda}}\\
M_y\ot C_\lambda \ar[rr]^{\varphi_y\ot C_\lambda}&& M'_{\eta(y)}\ot C_\lambda}
\end{equation}
\end{blanco}

\section{Algebras graded by a discrete Doi-Hopf datum}\selabel{2}
\begin{definition}\delabel{2.1}
Let $(G,\Lambda,X)$ be a discrete Doi-Hopf datum.
A $(G,\Lambda,X)$-graded algebra is an associative algebra $A$ (not necessarily with unit)
together with a direct sum decomposition
$$A=\oplus_{\lambda\in \Lambda}\oplus_{x\in X} A_{\lambda,x},$$
such that
\begin{equation}\eqlabel{2.1}
A_{\lambda,x}A_{\lambda',x'}\subset \delta_{x',x\gamma(\lambda')} A_{\lambda\lambda',x'},
\end{equation}
where $\delta$ is the Kronecker symbol. Moreover, for every $x\in X$, there exists a $1_x\in A_{e,x}$
such that 
\begin{eqnarray}
a1_x&=&a,~{\rm for~all}~\lambda\in\Lambda~{\rm and}~a\in A_{\lambda,x};\eqlabel{2.1.1}\\
1_xb&=&b,~{\rm for~all}~\lambda\in\Lambda~{\rm and}~b\in A_{\lambda,x\gamma(\lambda)}.\eqlabel{2.1.2}
\end{eqnarray}
\end{definition}

\begin{proposition}\prlabel{2.1a}
Let $A$ be a $(G,\Lambda,X)$-graded algebra, with either $G$ or $\Lambda$ a group, and put
$$A_{\lambda}=\oplus_{x\in X} A_{\lambda,x},$$
for all $\lambda\in \Lambda$. Then $A=\oplus_{\lambda\in \Lambda} A_\lambda$
is a $\Lambda$-graded algebra with idempotent local units. If $X$ is finite, then
$A$ is a $\Lambda$-graded algebra with unit $1=\sum_{x\in X} 1_x$.
\end{proposition}

\begin{proof}
It follows from \equref{2.1} that $A_{\lambda}A_{\lambda'}\subset A_{\lambda\lambda'}$.
If $G$ or $\Lambda$ is a group, then $\gamma(\lambda)$ is invertible in $G$, for all
$\lambda\in \Lambda$.\\
Take $a\in A_{\lambda,x}$. From  (\ref{eq:2.1}-\ref{eq:2.1.2}), it follows that
$$a1_y=\delta_{x,y}a~~{\rm and}~~1_ya=\delta_{y,x\gamma(\lambda)^{-1}}a.$$
In particular, $1_x1_y=\delta_{x,y}1_x$, so $\{1_x~|~x\in X\}$ is a set of orthgonal idempotents.
This implies the following: for any finite subset $I\subset X$, we have the following implications:
\begin{eqnarray*}
x\in I&\Longrightarrow& a(\sum_{y\in I}1_y)=a;\\
x\gamma(\lambda)^{-1}\in I&\Longrightarrow& (\sum_{y\in I}1_y)a=a.
\end{eqnarray*}
Now take a finite subset $B\subset A$. There exist $\lambda_1,\cdots,\lambda_n\in \Lambda$
and $x_1,\cdots,x_m\in X$ such that
$$B\subset \bigoplus_{i=1}^n\bigoplus_{j=1}^m A_{\lambda_i,x_j}.$$
We can always add $e$ to $\{\lambda_1,\cdots,\lambda_n\}$, so it is no restriction to assume that
$\lambda_1=e$.
Let $a$ be a homogeneous component of one of the elements of $B$. Then we find
$i\in \{1,\cdots n\}$ and $j\in \{1,\cdots, m\}$ such that
$a\in A_{\lambda_i,x_j}$. Consider
$I=\{x_j\gamma(\lambda_i)^{-1}~|~i\in \{1,\cdots n\},~j\in \{1,\cdots, m\}\}$.
Since $x_j=x_j\gamma(\lambda_1)^{-1},~x_j\gamma(\lambda_i)^{-1}\in I$, we have
$$a(\sum_{y\in I}1_y)=(\sum_{y\in I}1_y)a=a,$$
and then it follows that
$$b(\sum_{y\in I}1_y)=(\sum_{y\in I}1_y)b=b,$$
for all $b\in B$. If $X$ is finite, then the above arguments show that
$1=\sum_{x\in X} 1_x$ is a unit for $A$, and then $A$ is a unital $\Lambda$-graded
algebra.
\end{proof}

From \equref{2.1} and \equref{2.1.1}, we deduce that $A_{e,x}A_{e,x'}=\delta_{x,x'} A_{e,x'}$. Therefore every $A_{e,x}$
is a $k$-algebra, with unit $1_x$, and $(X, (A_{e,x})_{x\in X})$ is an algebra
in $\Tt_k$.\\
Assume that $X$ is finite, so that $A$ is a unital $\Lambda$-graded algebra.
Then $1\in A_e$, see for example \cite[Prop. I.1.1]{nvo}, and we can write
$1=\sum_{x\in X}1_x\in \oplus_{x\in X} A_{e,x}$.

\begin{proposition}\prlabel{2.1abis}
Take a discrete Doi-Hopf datum $(G,\Lambda,X)$, with $G$ or $\Lambda$ a group,
and assume that $X$ is finite.
The following assertions are equivalent:
\begin{enumerate}
\item[(a)] $A=\oplus_{\lambda\in \Lambda}\oplus_{x\in X} A_{\lambda,x}$ 
is a $(G,\Lambda,X)$-graded algebra;
\item[(b)] $A=\oplus_{\lambda\in \Lambda}\bigl(\oplus_{x\in X} A_{\lambda,x}\bigr)$ is 
a unital $\Lambda$-graded
algebra, and, with $1_x\in A_{e,x}$ defined as above,
\end{enumerate}
\begin{eqnarray}
A_{\lambda,x}1_{x'}&= & \delta_{x,x'} A_{\lambda,x};\eqlabel{2.1.4}\\
1_xA_{\lambda',x'}&=& \delta_{x',x\gamma(\lambda')}A_{\lambda',x'}.\eqlabel{2.1.5}
\end{eqnarray}
\end{proposition}

\begin{proof}
$\ul{(a)\Rightarrow (b)}$. We have already seen above that $A$ is a unital $\Lambda$-graded algebra;
(\ref{eq:2.1.4}-\ref{eq:2.1.5}) follow immediately from (\ref{eq:2.1}-\ref{eq:2.1.2}).\\
$\ul{(b)\Rightarrow (a)}$. 
Take $a\in A_{\lambda, x}$. It follows from \equref{2.1.4} that $a1_{x'}=0$ if $x\neq x'$. Therefore
$a=a1=\sum_{x'\in X} a1_{x'}=a1_x$, proving \equref{2.1.1}. \equref{2.1.2} is proved in a similar way:
let $b\in A_{\lambda,x\gamma(\lambda)}$. It follows from \equref{2.1.5} that $1_{x'}b=0$ if
$x'\neq x$, so $b=1b=\sum_{x'\in X} 1_{x'}b=1_x b$.\\
Let $a=1_x\in A_{e,x}$. It follows that $1_x1_{x'}=\delta_{x,x'}1_x$.\\
In order to show that \equref{2.1} holds, take $a\in A_{\lambda,x}$ and $b\in A_{\lambda',x'}$.
Then $ab\in A_{\lambda\lambda'}$, since $A$ is a $\Lambda$-graded algebra. Now we have that
$$ab=(a1_x)(1_{x'\gamma(\lambda')^{-1}}b)=a(1_x1_{x'\gamma(\lambda')^{-1}})b=0,$$
if $x\neq x'\gamma(\lambda')^{-1}$, or, equivalently, $x'\neq x\gamma(\lambda')$. Now let
$x'= x\gamma(\lambda')$. Since $A$ is a $\Lambda$-graded algebra, $a\in A_\lambda$ and
$b\in A_{\lambda'}$, we have that
$$ab\in A_{\lambda\lambda'}=\bigoplus_{y\in X} A_{\lambda\lambda',y}=
\bigoplus_{y\in X} A_{\lambda\lambda',y}1_y,$$
hence we have that $ab=\sum_{y\in X} c_y1_y$, with $c_y\in A_{\lambda\lambda',y}$.
Now $b1_{x'}=b$, hence 
$$ab=ab1_{x'}=\sum_{y\in X} c_y 1_y1_{x'}=c_{x'}1_{x'}\in A_{\lambda\lambda',x'}1_{x'}
=A_{\lambda\lambda',x'},$$
and this shows that \equref{2.1} holds.
\end{proof} 

\begin{blanco}\bllabel{2.2}{\bf 2-categorical interpretation.}
The definition of algebra graded by a discrete Doi-Hopf datum can be rephrased in terms of
2-categories. For more detail on 2-categories, we refer the reader to \cite[Ch. 7]{Borceux}.\\
To a discrete Doi-Hopf datum $(G,\Lambda,X)$, we associate a 2-category $\Gg$, under
the assumption that $G$ or $\Lambda$ is a group.
The objects of $\Gg$ are the elements of $X$, and the morphisms are the elements of
$\Lambda\times X$. $(\lambda,x)\in \Lambda\times X$ is a morphism with target $x$ and
source $x\gamma(\lambda)^{-1}$: $s(\lambda,x)=x\gamma(\lambda)^{-1}$ and $t(\lambda,x)=x$.
Then the composition $(\lambda',z)\circ (\lambda,y)$ is defined if and only
$y=t(\lambda, y)=s(\lambda',z)=z\gamma({\lambda'})^{-1}$, and, in this case
$(\lambda',z)\circ (\lambda,y)=(\lambda\lambda' ,z)$.
It is easy to verify that the identity morphism on $x$ is $(e,x)$.\\
Like every category, $\Gg$ can be viewed as a 2-category: the 0-cells are the objects of $\Gg$,
and, for all $x,y\in X$, $\Hom(x,y)$ is the discrete category with objects the morphisms $x\to y$
in $\Gg$. The only 2-cells are then the identity 2-cells. For every $x\in X$, we have
the unit functor $u_x:\ {\bf 1}\to \Hom(x,x)$, sending the object $0$ of ${\bf 1}$ to $(e,x)$,
and the morphism $1$ of ${\bf 1}$ to the identity of $(e,x)$. ${\bf 1}$ is the category with
one object $0$ and one morphism $1$.\\
The category of $k$-modules $\Mm_k$ is monoidal, so it can be viewed as a bicategory with
one object $*$. To simplify notation, we will treat $\Mm_k$ as if it were a strict monoidal category,
or a 2-category with one object. The unit functor $u_*:\ {\bf 1}\to \Mm_k$ sends $0$ to $k$ and
$1$ to the identity of $k$.
\end{blanco}

\begin{proposition}\prlabel{2.3}
Assume that $X$ is finite, and that $\Lambda$ or $G$ is a group.
Then we have a bijective correspondence between $(G,\Lambda,X)$-graded algebras
and lax functors $F:\ \Gg\to \Mm_k$.
\end{proposition}

\begin{proof}
According to  \cite[Def. 7.5.1]{Borceux}, a lax functor $F:\ \Gg\to \Mm_k$ consists of the following data:\\
(a) for every $x\in X$, a 0-cell $F(x)$ of $\Mm_k$. Since $\Mm_k$ has only one 0-cell, so there is
only one way to define $F$ at the level of 0-cells;\\
(b) for every $x,y\in X$, a functor $F_{x,y}:\ \Hom(x,y)\to \ul{\Hom}(F(x),F(y))=\Mm_k$.
Since $\Hom(x,y)$ is discrete, it suffices to give $F_{x\gamma(\lambda)^{-1},x}(\lambda,x)$, for every morphism
$(\lambda,x)$ in $\Gg$. Write $F_{x\gamma(\lambda)^{-1},x}(\lambda,x)=A_{(\lambda,x)}$.\\
(c) for $x,y,z\in X$, we have to give a natural transformation $\mu:\ F_{x,y}\Rightarrow F_{y,z}\to F_{x,z}$.
This means that for every $(\lambda,y)\in \Hom(x,y)$ and $(\lambda',z)\in \Hom(y,z)$, we have
to give a $k$-linear map
$$\mu_{(\lambda,y),(\lambda',z)}:\ A_{(\lambda,y)}\ot A_{(\lambda',z)}\to A_{(\lambda\lambda',z)}.$$
Since $y=z\gamma({\lambda'})^{-1}$, we find $k$-linear maps
$A_{(\lambda,y)}\ot A_{(\lambda',y\gamma(\lambda'))}\to A_{(\lambda\lambda',z)}$, which is precisely
what is needed to define on $A=\oplus_{(\lambda,x)\in \Lambda\times X} A_{(\lambda,x)}$ a multiplication that satisfies
\equref{2.1}. The naturality of $\mu$ is automatically fulfilled since $\Hom(x,y)$ is discrete. The associativity
of the multiplication on $A$ follows from the functorial properties of $F$.\\
(d) For all $x\in X$, we need a natural transformation
$$\delta_x:\ u_*\Rightarrow F_{x,x}\circ u_x.$$
This natural transformation is determined by a linear map
$$\delta_x(0):\ u_*(0)=k\to F_{x,x}(u_x(0))= A_{e,x}.$$
The diagrams (7.12) in \cite{Borceux} have to commute. In our particular situation, this means
that the diagrams
$$\xymatrix{
k\ot A_{\lambda,x\gamma(\lambda)}\ar[rr]^{\delta_x(0)\ot A_{\lambda,x\gamma(\lambda)}}
\ar[drr]_=&& A_{e,x}\ot A_{\lambda,x\gamma(\lambda)}
\ar[d]^{\mu_{(e,x),(\lambda,x\gamma(\lambda))}}\\
&&A_{\lambda,x\gamma(\lambda)}
}$$
$$\xymatrix{
A_{\lambda,x}\ot k \ar[rr]^{A_{\lambda,x}\ot \delta_x(0)}
\ar[drr]_=&& A_{\lambda,x}\ot A_{e,x}
\ar[d]^{\mu_{(\lambda,x),(e,x)}}\\
&&A_{\lambda,x}
}$$
commute. Now write $1_x=\delta_x(0)(1_k)$. The commutativity of the above diagrams is
equivalent to (\ref{eq:2.1.1}-\ref{eq:2.1.2}).
\end{proof}

\begin{blanco}{\bf The smash product.}\bllabel{2.5}
We propose a first method to construct algebras graded by a discrete Doi-Hopf datum
$(G,\Lambda,X)$, in the situation where $G$ or $\Lambda$ is a group.
Let $\ul{A}$ be a right $\ul{H}$-comodule algebra, as in \blref{1.5}.
Let $B$ be a $\Lambda$-graded algebra, and assume that every $B_\lambda$ is
a left $H_{\gamma(\lambda)^{-1}}$-module; the action of $h\in H_{\gamma(\lambda)^{-1}}$
on $b\in B_\lambda$ is denoted by $h\leftactr b$. Moreover, assume that
\begin{equation}\eqlabel{2.5.1}
h\leftactr (bb')= (h_{(2,\gamma(\lambda)^{-1})}\leftactr b)(h_{(1,\gamma(\lambda')^{-1})}\leftactr b'),
\end{equation}
for all $b\in B_\lambda$, $b'\in B_{\lambda'}$, $h\in H_{\gamma(\lambda\lambda')^{-1}}$, and
$h\leftactr 1=\varepsilon(h)1$, for all $h\in H_e$. Now define
$$B\# \ul{A}=\bigoplus_{\lambda\in \Lambda}\bigoplus_{x\in X} B_\lambda\# A_x.$$
Here $B_\lambda\# A_x= B_\lambda\ot A_x$ as a $k$-module. We define a multiplication map
on $B\# \ul{A}$, making it a $(G,\Lambda,X)$-graded algebra. We need to define multiplication
maps
$$(B_\lambda\# A_x)\ot (B_{\lambda'}\# A_{x'})\to B\#\ul{A}.$$
If $x'\neq x\gamma(\lambda')$, then we let this multiplication map be zero. For $x'=x\gamma(\lambda')$,
$$\mu:\ (B_\lambda\# A_x)\ot (B_{\lambda'}\# A_{x'})\to B_{\lambda\lambda'}\# A_{x'}$$
is given by the formula
$$(b\# a)(b'\# a')=b(a_{[1,\gamma(\lambda')^{-1}]}\leftactr b')\# a_{[0,x\gamma(\lambda')]}a'.$$
\end{blanco}

\begin{proposition}\prlabel{2.6}
With notation as above, $B\#\ul{A}$ is an algebra graded by $(G,\Lambda,X)$.
\end{proposition}

\begin{proof}
We have to show that the mulitplication is associative. Take $a\in A_x$, $a'\in A_{x'}$,
$a''\in A_{x''}$, $b\in B_{\lambda}$, $b'\in B_{\lambda'}$ and $b''\in B_{\lambda''}$. Also assume
that $x'=x\gamma(\lambda')$ and $x''=x'\gamma(\lambda'')$.
\begin{eqnarray*}
&&\hspace*{-15mm}
\bigl((b\#a)(b'\#a')\bigr)(b''\#a'')=
\bigl(b(a_{[1,\gamma(\lambda')^{-1}]}\leftactr b')\# a_{[0,x']}a'\bigr)(b''\#a'')\\
&=&
b(a_{[2,\gamma(\lambda')^{-1}]}\leftactr b')
\bigl((a_{[1,\gamma(\lambda'')^{-1}]}a'_{[1,\gamma(\lambda'')^{-1}]})\leftactr b''\bigr)
\# a_{[0,x'']}a'_{[0,x'']}a'';\\
&&\hspace*{-15mm}
(b\#a)\bigl((b'\#a')(b''\#a'')\bigr)
=(b\#a)\bigl(b'(a'_{[1,\gamma(\lambda'')^{-1}]}\leftactr b'')\# a'_{[0,x'']}a''\bigr)\\
&=&b\Bigl(a_{[1,\gamma(\lambda'\lambda'')^{-1}]}\leftactr \bigl(b'(a'_{[1,\gamma(\lambda'')^{-1}]}
\leftactr b'')\bigr)\Bigr) \# a_{[0,x'']} a'_{[0,x'']}a''\\
&=&
b(a_{[2,\gamma(\lambda')^{-1}]}\leftactr b')
\bigl((a_{[1,\gamma(\lambda'')^{-1}]}a'_{[1,\gamma(\lambda'')^{-1}]})\leftactr b''\bigr)
\# a_{[0,x'']}a'_{[0,x'']}a''.
\end{eqnarray*}
With the same notation, we easily compute that
\begin{eqnarray*}
&&\hspace*{-2cm}
(1\#1_x)(b'\# a')= 1(1_{\gamma(\lambda')^{-1}})\leftactr b')\# 1_{x'}a'=b'\#a';\\
&&\hspace*{-2cm}
(b\# a)(1\# 1_x)= b(a_{[1,\gamma(\lambda')^{-1}]}\leftactr 1)\# a_{[0,x']}1_{x'}\\
&=& b\# \varepsilon (a_{[1,\gamma(\lambda')^{-1}]})a_{[0,x']} =b\# a.
\end{eqnarray*}
\end{proof}

\begin{blanco}\bllabel{2.7} {\bf The Koppinen smash product.}
Now we introduce a second method to construct algebras graded by a discrete
Doi-Hopf datum. Let $(\ul{H},\ul{A},\ul{C})$ be a Doi-Hopf datum in $\Tt_k$, with
$\Lambda$ a group, and let
$$\Aa=\bigoplus_{\lambda\in G}\bigoplus_{x\in X} \Hom(C_{\lambda^{-1}},A_x)=
\bigoplus_{\lambda\in G}\bigoplus_{x\in X} \Aa_{\lambda,x}.$$
We define multiplication maps
$$ \Aa_{\lambda,x}\ot  \Aa_{\lambda',x'} \to \Aa;$$
for $x'\neq x\gamma(\lambda')$, this multiplication map is $0$. For
$x'= x\gamma(\lambda')$, then we describe
$$\mu:\ \Aa_{\lambda,x}\ot  \Aa_{\lambda',x'} \to \Aa_{\lambda\lambda',x'}.$$
For $f\in \Aa_{\lambda,x}$ and
$g\in \Aa_{\lambda',x'}$, $\mu(f\ot g)=f\# g\in \Aa_{\lambda\lambda',x'}$, 
is given by the following formula, for $c\in C_{(\lambda\lambda')^{-1}}$:
$$(f\# g)(c)=
f(c_{(2,\lambda^{-1})})_{[0,x']}g\bigl( c_{(1,(\lambda')^{-1})}f(c_{(2,\lambda^{-1})})_{[1,\gamma(\lambda')^{-1}]}
\bigr)\in A_{x'}.$$
\end{blanco}

\begin{proposition}\prlabel{2.8}
$\Aa$, as defined in \blref{2.7} is an algebra graded by $(G,\Lambda,X)$.
\end{proposition}

\begin{proof}
Take $f\in \Aa_{\lambda,x}$, $g\in \Aa_{\lambda',x'}$ and $h\in \Aa_{\lambda'',x''}$.
Assume also that $x'=x\gamma(\lambda')$ and $x''=x'\gamma(\lambda'')$. We have to
show that $(f\# g)\# h=f\# (g\# h)$. For $c\in C_{(\lambda\lambda'\lambda'')^{-1}}$, we have
\begin{eqnarray*}
&&\hspace*{-10mm}
((f\# g)\# h)(c)\\
&=&
\bigl((f\# g)(c_{(2,(\lambda\lambda')^{-1})}\bigr)_{[0,x'']}
h\Bigl( c_{(1,(\lambda'')^{-1})} \bigl((f\# g)(c_{(2,(\lambda\lambda')^{-1})}\bigr)_{[1,\gamma(\lambda'')^{-1}]}\Bigr)\\
&=& f(c_{(3,\lambda^{-1})})_{[0,x'']}
g\Bigl( c_{(2,(\lambda')^{-1})}f(c_{(3,\lambda^{-1})})_{[2,\gamma(\lambda')^{-1}]}\Bigr)_{[0,x'']}
h\Bigl( c_{(1,(\lambda'')^{-1})} \\
&&f(c_{(3,\lambda^{-1})})_{[1,\gamma(\lambda'')^{-1}]}
g\bigl( c_{(2,(\lambda')^{-1})}f(c_{(3,\lambda^{-1})})_{[2,\gamma(\lambda')^{-1}]}\bigr)_{[1,\gamma(\lambda'')^{-1}]}
\Bigr)\\
&=&
f(c_{(2,\lambda^{-1})})_{[0,x'']}
(g\# h)\bigl(c_{(1,(\lambda'\lambda'')^{-1})}f(c_{(2,\lambda^{-1})})_{[1,\gamma(\lambda'\lambda'')^{-1}]}\bigr)\\
&=&
(f\# (g\# h))(c).
\end{eqnarray*}
Define $e_x: C_e\to A_x$ by $e_x(c)=\varepsilon(c)1_x$. Then it is easy to compute that
$e_x\# g=g$ and $f\# e_{x'}=f$.
\end{proof}

\begin{blanco}\bllabel{2.9}
Let $(\ul{H},\ul{A},\ul{C})$ be a Doi-Hopf datum in $\Tt_k$, with $\Lambda$ a group, and put
$$B=\bigoplus_{\lambda\in \Lambda} B_\lambda,$$
with $B_\lambda= C^*_{\lambda^{-1}}$. In \blref{1.4ter}, we showed that $B$ is a $\Lambda$-graded
algebra.
$B_\lambda$ is a left $H_{\gamma(\lambda)^{-1}}$-module: for $\xi\in B_\lambda$,
$h\in H_{\gamma(\lambda)^{-1}} $and $c\in C_{\lambda^{-1}}$, let
$$(h\leftactr \xi)(c)=\xi(ch).$$
It is easy to verify that \equref{2.5.1} is satisfied: for $\xi'\in B_{\lambda'}$,
$h\in H_{\gamma(\lambda\lambda')^{-1}}$ and $c\in C_{(\lambda\lambda')^{-1}}$, we have
\begin{eqnarray*}
&&\hspace*{-2cm}
(h\leftactr (\xi\xi'))(c)=(\xi\xi')(ch)
=\xi\bigl((ch)_{(2,\lambda^{-1})}\bigr) \xi'\bigl((ch)_{(1,(\lambda')^{-1})}\bigr)\\
&=& \xi\bigl(c_{(2,\lambda^{-1})}h_{(2,\gamma(\lambda)^{-1})}\bigr)
\xi'\bigl(c_{(1,(\lambda')^{-1})}h_{(1,\gamma(\lambda')^{-1})}\bigr)\\
&=&\bigl((h_{(2,\gamma(\lambda)^{-1})}\leftactr \xi)(h_{(1,\gamma(\lambda')^{-1})}\leftactr\xi')\bigr)(c).
\end{eqnarray*}
Now we can consider the smash product $B\#\ul{A}$, as in \blref{2.5}. Consider the maps
$$\alpha_{\lambda,x}:
C^*_{\lambda^{-1}}\# A_x\to \Aa_{\lambda,x}=\Hom(C_{\lambda^{-1}},A_x),~~
\alpha_{\lambda,x}(\xi\# a)(c)=\xi(c)a,$$
for $\xi\in C^*_{\lambda^{-1}}$, $a\in A_x$, $c\in C_{\lambda^{-1}}$. It is well-known that
$\alpha_{\lambda,x}$ is an isomorphism of $k$-modules if $C_{\lambda}$ is finitely generated
and projective as a $k$-module. For later use, we describe $\alpha_{\lambda,x}^{-1}$,
using the notation introduced in \blref{1.4ter} for the dual basis of $C_\lambda$:
\begin{equation}\eqlabel{2.9.1}
\alpha_{\lambda,x}^{-1}(f)=\xi^{(\lambda^{-1})}\# f(c^{(\lambda^{-1})}).
\end{equation}
\end{blanco}

\begin{proposition}\prlabel{2.10}
With notation as in \blref{2.9}, 
$$\alpha=\bigoplus_{\lambda\in \Lambda}\bigoplus_{x\in X}\alpha_{\Lambda,x}:\
B\# \ul{A}\to \Aa$$
is a morphism of algebras graded by $(G,\Lambda,X)$. If every $C_\lambda$ is finitely
generated and projective as a $k$-module, then it is an isomorphism.
\end{proposition}

\begin{proof}
Take $\xi\in B_\lambda$, $\xi'\in B_{\lambda'}$, $a\in A_x$, $a'\in A_{x'}$, and assume that
$x'=x\gamma(\lambda')$. For all $c\in C_{(\lambda\lambda')^{-1}}$, we have
\begin{eqnarray*}
&&\hspace*{-2cm}
\alpha_{\lambda\lambda',x'}\bigl((\xi\# a)(\xi'\# a')\bigr)(c)
=\alpha_{\lambda\lambda',x'}\bigl(
\xi(a_{[1,\gamma(\lambda')^{-1}]}\leftactr \xi')\# (a_{[0,x']}a')\bigr)(c)\\
&=& \xi(c_{(2,\lambda^{-1})})\xi'(c_{(1,(\lambda')^{-1})}a_{[1,\gamma(\lambda')^{-1}]})a_{[0,x']}a'\\
&=&\bigl(\alpha_{\lambda,x}(\xi\# a)\alpha_{\lambda',x'}(\xi'\# a')\bigr)(c).
\end{eqnarray*}
It is also obvious that $\alpha_{e,x}(\varepsilon\# 1_x)=e_x$.
\end{proof}

\section{Modules graded by $(G,\Lambda,X)$-sets}\selabel{3}
\begin{definition}\delabel{3.0}
Let $(G,\Lambda,X)$ be a discrete Doi-Hopf datum, and $A$ a $(G,\Lambda,X)$-graded algebra.
Let $Y$ be a $(G,\Lambda,X)$-set, see \blref{1.2}. A right $A$-module $M$ is graded by
the $(G,\Lambda,X)$-set $Y$ if
$$M=\bigoplus_{y\in Y} M_y$$
with
\begin{equation}\eqlabel{3.1.1}
M_y A_{\lambda,x}\subset \delta_{x,\beta(y\lambda)}M_{y\lambda}
\end{equation}
and 
\begin{equation}\eqlabel{3.1.0}
m1_{\beta(y)}=m,
\end{equation}
for all $m\in M_y$.
\end{definition}

\begin{example}\exlabel{3.0}
Let $Y$ be a $(G,\Lambda,X)$-set. $Z\subset Y$ is a $(G,\Lambda,X)$-subset of $Y$ if
$z\lambda\in Z$, for all $\lambda\in \Lambda$ and $z\in Z$.\\
Now suppose that $M=\oplus_{y\in Y} M_y$ is a right $A$-module graded by the 
 $(G,\Lambda,X)$-set $Y$. Then $N=\oplus_{z\in Z} M_z$ is a right $A$-module graded by the 
 $(G,\Lambda,X)$-set $Z$. Indeed, for all $z\in Z$, $\lambda\in \Lambda$ and $x\in X$, we have
 $$N_zA_{\lambda,x}=M_zA_{\lambda,x}\subset \delta_{x,\beta(z\lambda)}M_{z\lambda}
 =\delta_{x,\beta(z\lambda)}N_{z\lambda}.$$
 \end{example}

\begin{example}\exlabel{3.1}
Recall from \blref{1.2} that $Y=\Lambda\times X$ is a $(G,\Lambda,X)$-set.
Let $A$ be a $(G,\Lambda,X)$-graded algebra; then $A$ viewed as a right $A$-module
is graded by the $(G,\Lambda,X)$-set $\Lambda\times X$. We need to verify that
\begin{equation}\eqlabel{3.1.2}
A_{\lambda,x}A_{\lambda',x'}\subset \delta_{x',\beta((\lambda,x)\lambda')}A_{(\lambda,x)\lambda'}.
\end{equation}
We have seen in \blref{1.2} that $(\lambda,x)\lambda'=(\lambda\lambda',x\gamma(\lambda'))$
and $\beta((\lambda,x)\lambda')=x\gamma(\lambda')$, and then \equref{3.1.2} reduces to
\equref{2.1}. It is also easy to check the unit condition: for $a\in A_{\lambda,x}$, we have
that $a1_{\beta(\lambda,x)}=a1_x=a$.\\
Now fix $x\in X$. Then 
$$Z_x=\{(\lambda,x\gamma(\lambda))~|~\lambda\in \Lambda\}$$
is a $(G,\Lambda,X)$-subset of $\Lambda\times X$. Indeed, for all $\lambda'\in \Lambda$,
$(\lambda,x\gamma(\lambda))\lambda'=(\lambda\lambda',x\gamma(\lambda\lambda'))
\in Z_x$. It follows from \exref{3.0} that
$A^{(x)}=\oplus_{\lambda\in \Lambda} A_{\lambda,x\gamma(\lambda)}$
is a right $A$-module graded by the $G$-set $Z_x$.
\end{example}

Assume now that $X$ is finite; then we know that $A$ is an algebra with unit $1=\sum_{x\in X} 1_x$. 
If $M$ is a right $A$-module graded by a $(G,\Lambda,X)$-set $Y$, then we have for all
$m\in M_y$ that $m1_x=0$ if $x\neq \beta(y)$, hence $m1=\sum_{x\in X} m1_x=m1_{\beta(y)}=m$,
so $M$ is a unital $A$-module. It also follows from \equref{3.1.1} that $M_yA_{\lambda}
\subset M_{y\lambda}$, hence $M$ is a right $A$-module graded by the $\Lambda$-set $Y$. We refer to 
\cite{NastasescuRV90} for a discussion of modules graded by $G$-sets.\\
Conversely, let $M$ be a right $A$-module graded by a $\Lambda$-set $Y$ (which is not
necessarily a $(G,\Lambda,X)$-set). Since $1=\sum_{x\in X} 1_x$, we have, for all $y\in Y$,
$M_y=M_y1=\sum_{x\in X} M_y1_x$.
Let $x\neq x'\in X$, and assume that $m\in M_y1_x\cap M_y1_{x'}$. Then $m=n1_{x'}$ for some
$n\in M$, and $m=m1_x=n1_{x'}1_x=0$. Hence $M_y1_x\cap M_y1_{x'}=\{0\}$ and
\begin{equation}\eqlabel{3.1.3}
M_y=\oplus_{x\in X} M_y1_x.
\end{equation}
If $M$ is graded by a $(G,\Lambda,X)$-set $Y$, then it follows from \equref{3.1.1} that
$M_y1_x=\{0\}$ if $x\neq \beta(y)$, and then we find that $M_y1_{\beta(y)}=M_y$. Hence
at most one direct summand in \equref{3.1.3} is nontrivial.

\begin{proposition}\prlabel{3.2}
Let $A$ be a $(G,\Lambda,X)$-graded algebra, with $X$ finite,
and $Y$ a $(G,\Lambda,X)$-set. For a (unital) right $A$-module $M=\oplus_{y\in Y}M_y$, the
following assertions are equivalent
\begin{itemize}
\item $M$ is graded by the $(G,\Lambda,X)$-set $Y$;
\item $M$ is graded by the $\Lambda$-set $Y$ and $M_y1_x= \delta_{x,\beta(y)}M_y$.
\end{itemize}
\end{proposition}

\begin{proof}
$\ul{1)\Rightarrow 2)}$: see the arguments preceding \prref{3.2}.\\
$\ul{2)\Rightarrow 1)}$. Take $m\in M_y$. If $x\neq \beta(y)$, then $m1_x=0$, hence
$m=m1=\sum_{x\in X} m1_x=
m1_{\beta(y)}$.\\
Take $m\in M_y$ and $a\in A_{\lambda,x}$. Then 
$$ma=(m1_{\beta(y)})a=m(1_{\beta(y)}a)=m(\delta_{x,\beta(y)\gamma(\lambda)}a).$$
If $x\neq \beta(y)\gamma(\lambda)$, then $ma=0$. In any case $ma\in M_{y\lambda}$, so we
conclude that \equref{3.1.1} holds, since $\beta(y)\gamma(\lambda)=\beta(y\lambda)$.
\end{proof}

Now we introduce the category $\Zz_A^{(G,\Lambda,X)}$ of right $A$-modules graded 
by $(G,\Lambda,X)$-sets. The objects are couples $(Y,M)$, where $Y$ is a $(G,\Lambda,X)$-set,
and $M$ is a right $A$-module graded by the $(G,\Lambda,X)$-set $Y$. 
A morphism $(Y,M)\to (Y',M')$ in $\Zz_A^{(G,\Lambda,X)}$ is a couple $(\eta, \varphi)$, where
$\eta:\ Y\to Y'$ is a morphism of $(G,\Lambda,X)$-sets, and
$\varphi:\ M\to M'$ is a right $A$-linear map such that $\varphi(M_y)\subset M'_{\eta(y)}$.
If the condition $\varphi(M_y)\subset M'_{\eta(y)}$ is satisfied, then the condition that 
$\varphi$ is right $A$-linear is equivalent to the commutativity of the diagrams
\begin{equation}\eqlabel{3.2.5}
\xymatrix{
M_y\ot A_{\lambda,\beta(y\lambda)}\ar[d]_{\varphi_y\ot {\rm id}}\ar[rr]&&
M_{y\lambda}\ar[d]^{\varphi_{y\lambda}\hspace{8mm}.}\\
M'_{\eta(y)}\ot A_{\lambda,\beta'(\eta(y)\lambda)}\ar[rr]&&
M'_{\eta(y\lambda)}=M'_{\eta(y)\lambda}}
\end{equation}

$\Tt_A^{(G,\Lambda,X)}$ has the same objects as $\Zz_A^{(G,\Lambda,X)}$.
A morphism $(Y,M)\to (Y',M')$ in $\Tt_A^{(G,\Lambda,X)}$ is a couple $(\eta, (\varphi_{y'})_{y'\in Y'})$, where
$\eta:\ Y'\to Y$ is a morphism of $(G,\Lambda,X)$-sets, and $\varphi_{y'}:\ M_{\eta(y')}\to M'_{y'}$
are $k$-linear maps such that the diagram
\begin{equation}\eqlabel{3.2.6}
\xymatrix{
M_{\eta(y')}\ot A_{\lambda,\beta(\eta(y')\lambda)}\ar[d]_{\varphi_{y'}\ot {\rm id}}\ar[rr]&&
M_{\eta(y')\lambda}=M_{\eta(y'\lambda)}\ar[d]^{\varphi_{y'\lambda}}\\
M_{y'}\ot A_{\lambda,\beta'(y'\lambda)}\ar[rr]&&M'_{y'\lambda}}
\end{equation}
commutes, for all $y'\in Y'$ and $\lambda\in \Lambda$.\\
Observe that these definitions are designed in such a way that we have forgetful functors
$$\Zz_A^{(G,\Lambda,X)}\to \Zz_k~~{\rm and}~~\Tt_A^{(G,\Lambda,X)}\to \Tt_k.$$

\begin{proposition}\prlabel{3.3}
Let $(\ul{H},\ul{A},\ul{C})$ be a Doi-Hopf datum in $\Tt_k$. Then we have fully faithful functors
$$T:\ \Tt_k(\ul{H})_{\ul{A}}^{\ul{C}}\to \Tt_\Aa^{(G,\Lambda,X)}~~{\rm and}~~
Z:\ \Zz_k(\ul{H})_{\ul{A}}^{\ul{C}}\to \Zz_\Aa^{(G,\Lambda,X)}.$$
At the level of objects, the functors are defined in the same way: 
$T(\ul{M})=Z(\ul{M})=(Y, \oplus_{y\in Y}M_y)$, with multiplication maps
$M_y\ot \Aa_{\lambda,\beta(y\lambda)}\to M_{y\lambda}$ given by the formula
\begin{equation}\eqlabel{3.3.1}
mf=m_{[0,y\lambda]}f(m_{[1,\lambda^{-1}]}).
\end{equation}
At the level of morphisms, $T$ and $Z$ are defined by
$$T(\eta,(\varphi_{y'})_{y'\in Y'})=(\eta,(\varphi_{y'})_{y'\in Y'})~~{\rm and}~~
Z(\eta,(\varphi_y)_{y\in Y})=(\eta,\bigoplus_{y\in Y} \varphi_y).$$
\end{proposition}

\begin{proof}
We will show that the action \equref{3.3.1} is associative and satisfies the unit property.
Take $f\in \Aa_{\lambda,x}$, $f'\in \Aa_{\lambda',x'}$, with $x'=x\gamma(\lambda')$, so that
$f\# f'\in \Aa_{\lambda\lambda',x'}$. Let $m\in M_y$, and take $x=\beta(y)$. Now
\begin{eqnarray*}
(mf)f'&=&(m_{[0,y\lambda]}f(m_{[1,\lambda^{-1}]}))f'\\
&=& m_{[0,x']}f(m_{[2,\lambda^{-1}]})_{[0,x']} f'\bigl(m_{[1,(\lambda')^{-1}]}
f(m_{[2,\lambda^{-1}]})_{[1,\gamma(\lambda')^{-1}]}\bigr)\\
&=& m(f\#f').
\end{eqnarray*}
The unit property is handled as follows. For $m\in M_y$, we have
$$me_{\beta(y)}=m_{[0,y]}e_{\beta(y)}(m_{[1,e]})=m_{[0,y]}\varepsilon(m_{[1,e]})1_{\beta(y)}=m.$$
Now we look at the morphisms. Let $(\eta,(\varphi_{y'})_{y'\in Y'})$ be a morphism $\ul{M}\to
\ul{M}'$ in $\Tt_k(\ul{H})_{\ul{A}}^{\ul{C}}$. We then have to show that it is also a morphism
$(Y,\oplus_{y\in Y}M_y)\to (Y',\oplus_{y'\in Y'}M'_{y'})$ in $\Zz_\Aa^{(G,\Lambda,X)}$.
To this end, it suffices to show that the diagrams \equref{3.2.6} commute. Take
$m\in M_{\eta(y')}$ and $f\in \Aa_{\lambda,\beta(\eta(y')\lambda)}$. Then
\begin{eqnarray*}
&&\hspace*{-2cm}
\varphi_{y'}(m)f=\varphi_{y'}(m)_{[0,y'\lambda]}f(\varphi_{y'}(m)_{[1,\lambda^{-1}]})
\equal{\equref{1.4.2}}\varphi_{y'\lambda}(m_{[0,\eta(y')\lambda]})f(m_{[1,\lambda^{-1}]})\\
&=&\varphi_{y'\lambda}\bigl(m_{[0,\eta(y')\lambda]}f(m_{[1,\lambda^{-1}]})\bigr)=\varphi_{y'\lambda}(mf).
\end{eqnarray*}
Finally, take a morphism $(\eta,(\varphi_y)_{y\in Y}):\ \ul{M}\to
\ul{M}'$ in $\Zz_k(\ul{H})_{\ul{A}}^{\ul{C}}$. We have to show that $(\eta, \oplus_{y\in Y}$
is a morphism in $\Zz_\Aa^{(G,\Lambda,X)}$. To this end, we have to show that \equref{3.2.5} commutes.
For $m\in M_y$ and $f\in \Aa_{\lambda,\beta(y\lambda)}$, we have
\begin{eqnarray*}
&&\hspace*{-2cm}
\varphi_y(m)f=\varphi_y(m)_{[0,\eta(y\lambda)]}f(\varphi_y(m)_{[1,\lambda^{-1}]})
\equal{\equref{1.7.1}}
\varphi_{y\lambda}(m_{[0,y\lambda]})f(m_{[1,\lambda^{-1}]})\\
&=& \varphi_{y\lambda}\bigl(m_{[0,y\lambda]}f(m_{[1,\lambda^{-1}]})\bigr)
=\varphi_{y\lambda}(mf).
\end{eqnarray*}
\end{proof}

\begin{theorem}\thlabel{3.4}
Let $(\ul{H},\ul{A},\ul{C})$ be a Doi-Hopf datum in $\Tt_k$, and assume that every $C_\lambda$
is finitely generated and projective as a $k$-module. Then the functors $T$ and $Z$ from
\prref{3.3} are isomorphisms of categories.
\end{theorem}

\begin{proof}
We will construct a functor ${G}:\ \Tt_\Aa^{(G,\Lambda,X)}\to \Tt_k(\ul{H})_{\ul{A}}^{\ul{C}}$ and show
that it is the inverse of $T$. Take $(Y,M)\in \Mm_\Aa^{(G,\Lambda,X)}$. Let
$G(M)=(Y,(M_y)_{y\in Y})$, with structure described as below.\\
a) $M_y$ is a right $A_{\beta(y)}$-module:
$ma=m\alpha_{e,\beta(y)}(\varepsilon\# a)$,
for $m\in M_y$, $a\in A_{\beta(y)}$. Let $x=\beta(y)$. It is straightforward to see that this action
is associative.\\
b) Coaction maps $\rho_{y,\lambda}:\ M_{y\lambda}\to M_y\ot C_{\lambda}$ are defined as follows:
$$\rho_{y,\lambda}(m)=m\alpha_{\lambda^{-1},\beta(y)}(\xi^{(\lambda)}\# 1_{\beta(y)})\ot c^{(\lambda)},$$
where we use the notation introduced in \blref{1.4ter}. We have to show that this coaction is coassociative.
For $m\in M_{y\lambda\lambda'}$, we have that
\begin{eqnarray*}
&&\hspace*{-15mm}
(\rho_{y,\lambda}\ot C_{\lambda'})(\rho_{y\lambda,\lambda'}(m))\\
&=&(\rho_{y,\lambda}\ot C_{\lambda'})
\bigl(m\alpha_{(\lambda')^{-1},\beta(y\lambda)}(\xi^{(\lambda')}\# 1_{\beta(y\lambda)})\bigr)\ot c^{(\lambda')}\\
&=& m\alpha_{(\lambda')^{-1},\beta(y\lambda)} (\xi^{(\lambda')}\# 1_{\beta(y\lambda)})
\alpha_{\lambda^{-1},\beta(y)}(\xi^{(\lambda)}\#1_{\beta(y)})\ot c^{(\lambda)}\ot c^{(\lambda')}\\
&=&m\alpha_{(\lambda\lambda')^{-1},\beta(y)}(\xi^{(\lambda')}\xi^{(\lambda)}\# 1_{\beta(y)})
\ot c^{(\lambda)}\ot c^{(\lambda')};\\
&&\hspace*{-15mm}
(M_y\ot C_{\lambda,\lambda'})(\rho_{y,\lambda\lambda'}(m))\\
&=&m\alpha_{(\lambda\lambda')^{-1},\beta(y)}(\xi^{(\lambda\lambda')}\# 1_{\beta(y)})\ot \Delta_{\lambda,\lambda'}(c^{(\lambda\lambda')}).
\end{eqnarray*}
These expressions are equal since
\begin{eqnarray*}
&&\hspace*{-2cm}
\xi^{(\lambda')}\xi^{(\lambda)} \ot c^{(\lambda)}\ot c^{(\lambda')}
=(\xi^{(\lambda')}\xi^{(\lambda)} )(c^{(\lambda\lambda')})\xi^{(\lambda\lambda')}
\ot c^{(\lambda)}\ot c^{(\lambda')}\\
&=& \xi^{(\lambda')}(c^{(\lambda\lambda')}_{(2,\lambda')})\xi^{(\lambda)}(c^{(\lambda\lambda')}_{(1,\lambda)})
\xi^{(\lambda\lambda')}\ot c^{(\lambda)}\ot c^{(\lambda')}\\
&=&\xi^{(\lambda\lambda')}\ot 
 \xi^{(\lambda)}(c^{(\lambda\lambda')}_{(1,\lambda)})
c^{(\lambda)}
\ot \xi^{(\lambda')}(c^{(\lambda\lambda')}_{(2,\lambda')})c^{(\lambda')}\\
&=& \xi^{(\lambda\lambda')}\ot c^{(\lambda\lambda')}_{(1,\lambda)}\ot c^{(\lambda\lambda')}_{(2,\lambda')}
=\xi^{(\lambda\lambda')}\ot \Delta_{\lambda,\lambda'}(c^{(\lambda\lambda')}).
\end{eqnarray*}
Let us prove that the counit property holds. For $m\in M_y$, we have
$$(M_y\ot \varepsilon)\rho_{y,e}(m)=
m\alpha_{e,\beta(y)}(\xi^{(e)}\# 1_{\beta(y)})\varepsilon(c^{(e)})=
m\alpha_{e,\beta(y)}(\varepsilon \# 1_{\beta(y)})=m.$$
Finally, we need to prove that the action and coaction on $\ul{M}$ are compatible, that is,
$$\rho_{y,\lambda}(ma)=m_{[0,y]}a_{[0,\beta(y)]}\ot m_{[1,\lambda]}a_{[1,\gamma(\lambda)]},$$
for $m\in M_{y\lambda}$ and $a\in A_{\beta(y\lambda)}$.
\begin{eqnarray*}
&&\hspace*{-2cm}
\rho_{y,\lambda}(ma)=ma \alpha_{\lambda^{-1},\beta(y)}(\xi^{(\lambda)}\# 1_{\beta(y)})\ot c^{(\lambda)}\\
&=&
m\alpha_{e,\beta(y\lambda)}(\varepsilon\# a)\alpha_{\lambda^{-1},\beta(y)}(\xi^{(\lambda)}\# 1_{\beta(y)})\ot c^{(\lambda)}\\
&=&
m\alpha_{\lambda^{-1},\beta(y)}\bigl((\varepsilon\# a)(\xi^{(\lambda)}\# 1_{\beta(y)})\bigr)\ot c^{(\lambda)}\\
&=&
m\alpha_{\lambda^{-1},\beta(y)}\bigl( (a_{[1,\gamma(\lambda)]}\leftactr \xi^{(\lambda)})
\# a_{[0,\beta(y)]}\bigr)\ot c^{(\lambda)}\\
&\equal{\equref{2.9.2}}&
m\alpha_{\lambda^{-1},\beta(y)}\bigl( \xi^{(\lambda)}
\# a_{[0,\beta(y)]}\bigr)\ot c^{(\lambda)}a_{[1,\gamma(\lambda)]}\\
&=& m\alpha_{\lambda^{-1},\beta(y)}(\xi^{(\lambda)}\# 1_{\beta(y)})\alpha_{e,\beta(y)}
(\varepsilon\# a_{[0,\beta(y)]})\ot c^{(\lambda)}a_{[1,\gamma(\lambda)]}\\
&=&m_{[0,y]}a_{[0,\beta(y)]}\ot m_{[1,\lambda]}a_{[1,\gamma(\lambda)]}.
\end{eqnarray*}
Let us now show that $T$ and $G$ are inverses. Take $(Y,(M_y)_{y\in Y})\in \Tt_k(\ul{H})_{\ul{A}}^{\ul{C}}$.
Then $GT(Y,(M_y)_{y\in Y})=(Y,(M_y)_{y\in Y})$, where every $M_y$ is a right $A_{\beta(y)}$-module;
this new action is denoted $\cdot$, and we prove that it coincides with the original one:
for $m\in M_y$ and $a\in A_{\beta(y)}$, we have
$$
m\cdot a=m\alpha_{e,\beta(y)}(\varepsilon\# a)=
m_{[0,y]}(\alpha_{e,\beta(y)}(\varepsilon\# a))(m_{[1,e]})=
m_{[0,y]}\varepsilon(m_{[1,e]})a=ma.$$
We also have to show that the coaction maps $\tilde{\rho}_{y,\lambda}$ on $GT(Y,(M_y)_{y\in Y})$
coincide with the original ${\rho}_{y,\lambda}$ on $(Y,(M_y)_{y\in Y})$. For all $m\in M_{y\lambda}$,
we have
\begin{eqnarray*}
\tilde{\rho}_{y,\lambda}(m)&=&
m\alpha_{\lambda^{-1},\beta(y)}(\xi^{(\lambda)}\# 1_{\beta(y)})\ot c^{(\lambda)}\\
&=& m_{[0,y]}\alpha_{\lambda^{-1},\beta(y)}(\xi^{(\lambda)}\# 1_{\beta(y)})(m_{[1,\lambda]})\ot c^{(\lambda)}\\
&=& m_{[0,y]} \xi^{(\lambda)}(m_{[1,\lambda]}) 1_{\beta(y)}\ot c^{(\lambda)}
=m_{[0,y]}\ot m_{[1,\lambda]}=\rho_{y,\lambda}(m).
\end{eqnarray*}
Now let $(Y,M)\in \Tt_\Aa^{(G,\Lambda,X)}$. Then $TG(Y,M)=(Y,M)$, with new right $\Aa$-action
denoted $\cdot$. In order to show that this new action coincides with the original one, it suffices to
show that $m\cdot f=mf$, for all $m\in M_y$ and $f\in \Aa_{\lambda,\beta(y\lambda)}$ of the form
$$f=\alpha_{\lambda,\beta(y\lambda)}(\xi\# a),$$
where $\xi\in C^*_{\lambda^{-1}}$ and $a\in A_{\beta(y\lambda)}$.
\begin{eqnarray*}
&&\hspace*{-10mm} m\cdot f= m_{[0,y\lambda]}f(m_{[1,\lambda^{-1}]})
= m\alpha_{\lambda,\beta(y\lambda)}(\xi^{(\lambda^{-1})}\# 1_{\beta(y)}) f(c^{(\lambda^{-1})})\\
&=& m\alpha_{\lambda,\beta(y\lambda)}(\xi^{(\lambda^{-1})}\# 1_{\beta(y)}) \xi(c^{(\lambda^{-1})})a
= m\alpha_{\lambda,\beta(y\lambda)}(\xi\# 1_{\beta(y)})\alpha_{e,\beta(y\lambda)}(\varepsilon\# a)\\
&=& m\alpha_{\lambda,\beta(y\lambda)}(\xi\# a)=mf.
\end{eqnarray*}
It is left to the reader to show the result at the level of morphisms. The inverse $H$ of $Z$ is constructed
in a similar way.
\end{proof}

\section{Yetter-Drinfeld modules and the Drinfeld double}\selabel{4}
\begin{blanco}\bllabel{4.1} {\bf Crossed $G$-sets.}
Let $G$ be a group. Recall that a right crossed $G$-set is a $G$-set $V$ together with a map
$\nu:\ V\to G$ such that
$$\nu(vg)=g^{-1}\nu(v)g=\nu(v)^{g},$$
for all $v\in V$ and $g\in G$. This notion goes back to Whitehead, and it can be reformulated as
follows. Observe first that $G$ is a right $G\times G$-set, with action
$$l\cdot (g,g')=g^{-1}lg'.$$
The diagonal map $\gamma:\ G\to G\times G$ is clearly a morphism of monoids. Hence
$\GG=(G\times G,G,G)$ is a discrete Doi-Hopf datum. Then it is easy to see that a right crossed
$G$-set is the same thing as a $\GG$-set.
The category $\Xx_G^G$ of right crossed $G$-modules is a braided monoidal category:
for two crossed $G$-modules $(V,\nu)$ and $(V,\nu')$, $(V\times V',\omega)$, with
$\omega(v,v')=\nu(v)\nu'(v')$, and $(v,v')g=(vg,v'g)$ is again a crossed $G$-set. The unit object
is the singleton $\{*\}$, as a trivial right $G$-set, together with the map sending $*$ to
the unit element $e\in G$.\\
The braiding $c_{V,V'}: V\times V'\to V'\times V$ and its inverse are given by the following formulas:
$$c_{V,V'}(v,v')=(v',v\nu'(v'))~~;~~c_{V,V'}^{-1}(v',v)=(v\nu'(v')^{-1},v').$$
This can be verified directly, see \cite{FY} or \cite[XIII.1.4]{k}. 
It is also a consequence of the (folklore) fact that the category of
crossed $G$-sets can be obtained from the category of $G$-sets using the centre
construction, see \cite[Sec. 4]{cl} for a detailed explanation.
\end{blanco}

\begin{blanco}\bllabel{4.2a} {\bf Hopf group coalgebras.} 
Recall that a Hopf group coalgebra is a semi-Hopf group coalgebra $\ul{H}$ (as in \blref{1.5}),
such that the underlying monoid $G$ is a group, together with maps
$S_g,~\ol{S}_g:\ H_{g^{-1}}\to H_g$ $(g\in G)$ such that
\begin{eqnarray*}
S_g(h_{(1,g^{-1})})h_{(2,g)}&=& h_{(1,g)}S_g(h_{(2,g^{-1})})=\varepsilon(h)1_g,\\
h_{(2,g)}\ol{S}_g(h_{(1,g^{-1})})&=&\ol{S}_g(h_{(2,g^{-1})})h_{(1,g)}=\varepsilon(h)1_g,
\end{eqnarray*}
for all $g\in G$ and $h\in H_e$. The $S_g$ are called the antipode maps, while the
$\ol{S}_g$ are called the twisted antipode maps. The $\ol{S}_g$ are then the antipode maps
of the opposite Hopf group coalgebra $\ul{H}^{\rm op}$, which is defined as follows:
$H^{\rm op}_g=H_g$, with opposite multiplication, and $\Delta^{\rm op}_{g,g'}=\Delta_{g,g'}$. 
For all $g\in G$, $\ol{S}_g$ is the inverse of $S_{g^{-1}}$ and, according to \cite{vire}, they 
always exist in the case when each $H_g$ is finite dimensional ($G$ is arbitrary).\\
\end{blanco}

\begin{blanco}\bllabel{4.2} {\bf Yetter-Drinfeld modules.}
Let $\ul{H}$ be a semi-Hopf group coalgebra. Right-right $\ul{H}$-Yetter-Drinfeld modules were
introduced in \cite[Def. 4.4]{cl}. We recall this definition in the special case where
$\ul{H}$ is a Hopf group coalgebra. We need an object $\ul{M}=(V,(M_v)_{v\in V})\in \Tt_k$,
with $V$ a crossed right $G$-set ($G$ a group), together with the following structure:
\begin{itemize}
\item every $M_v$ is a right $H_{\nu(v)}$-module;
\item $\ul{M}$ is a right $\ul{H}$-comodule, with coaction maps $\rho_{v,g}:\ M_{vg}\to M_v\ot H_g$.
\end{itemize}
The following compatibility condition has to be satisfied
\begin{equation}\eqlabel{4.2.1}
\rho_{v,g}(mh)=m_{[0,v]}h_{(2,\nu(v))}\ot S_g(h_{(1,g^{-1})})m_{[1,g]}h_{(3,g)},
\end{equation}
for all $m\in M_{vg}$ and $h\in H_{\nu(vg)}=H_{g^{-1}\nu(v)g}$. 
$\YDT^{\ul{H}}_{\ul{H}}$ is the category of right-right $\ul{H}$-Yetter-Drinfeld modules and morphisms
that are morphisms in $\Tt^{\ul{H}}$ and $\Tt_{\ul{H}}$.\\
The category $\YDZ^{\ul{H}}_{\ul{H}}$ is introduced in a similar way; the objects coincide with the
objects of $\YDT^{\ul{H}}_{\ul{H}}$,
and the morphisms have
to be morphisms $\Zz^{\ul{H}}$ and $\Zz_{\ul{H}}$.
\end{blanco}

\begin{blanco}\bllabel{4.3} {\bf A Doi-Hopf datum.}
Let $\ul{H}$ be a Hopf group coalgebra. Then $\ul{H}^{\rm op}\ot \ul{H}$ is also a Hopf group coalgebra.
Then $\ul{H}$ is a right $\ul{H}^{\rm op}\ot \ul{H}$-comodule algebra, with structure maps
$$\rho_{l,(g,g')}:\ H_{g^{-1}lg'}\to H_l\ot (H_g\ot H_{g'})$$
given by
$$\rho_{l,(g,g')}(h)=h_{(2,l)}\ot S_g(h_{(1,g^{-1})})\ot h_{(3,g')}.$$
A technical but straightforward computation shows that the coassociativity and counit properties
hold.\\
$\ul{H}$ is a right $\ul{H}^{\rm op}\ot\ul{H}$-module coalgebra. Indeed, for $g\in G$, $\gamma(g)=
g\ot g$ and $H_g$ is a right $H^{\rm op}_g\ot H_g$-module, with action
$k(h\ot h')=hkh'$.\\
We conclude that $(\ul{H}^{\rm op}\ot\ul{H},\ul{H},\ul{H})$ is a Doi-Hopf datum in $\Tt_k$.
\end{blanco}

\begin{proposition}\prlabel{4.4}
For a Hopf group coalgebra $\ul{H}$, the categories $\YDT^{\ul{H}}_{\ul{H}}$
(resp. $\YDZ^{\ul{H}}_{\ul{H}}$) and
$\Tt_k(\ul{H}^{\rm op}\ot\ul{H})^{\ul{H}}_{\ul{H}}$ (resp. $\Zz_k(\ul{H}^{\rm op}\ot\ul{H})^{\ul{H}}_{\ul{H}}$)
 are isomorphic.
\end{proposition}

\begin{proof}
Objects in $\Tt_k(\ul{H}^{\rm op}\ot\ul{H})^{\ul{H}}_{\ul{H}}$ and $\YDT^{\ul{H}}_{\ul{H}}$
are objects $\ul{M}\in \Tt_k$ with a right $\ul{H}$-action and a right $\ul{H}$-coaction.
We have to show that the compatibility relations in both categories are the same.\\
Let $\ul{M}=(V,(M_v)_{v\in V})\in \Tt_k$, and assume that $V$ is a right crossed $G$-set,
$M_v$ is a right $H_{\nu(v)}$-module, for all $v\in V$, and $\ul{\rho}:\ \ul{M}\to
\ul{M}\ot \ul{H}$ is a right $\ul{H}$-coaction. Then we have maps $\rho_{v,g}:\
M_{vg}\to M_v\ot H_g$. The compatibility relation \equref{1.4.1} now takes the following
form: for all $m\in M_{vg}$ and $h\in H_{\nu(vg)}$, we have
\begin{equation}\eqlabel{4.4.1}
\rho_{v,g}(mh)=m_{[0,v]}h_{[0,\nu(v)]}\ot m_{[1,g]}h_{[1,\gamma(g)]}.
\end{equation}
Now
$$\rho_{\nu(v),\gamma(g)}(h)=h_{(2,\nu(v))}\ot(S_g(h_{(1,g^{-1})})\ot h_{(3,g)}),$$
so \equref{4.4.1} is equivalent to \equref{4.2.1}, as needed. The statement at the level
of morphisms is left to the reader.
\end{proof}

\begin{blanco}\bllabel{4.5}
Now assume that $H_\lambda$ is finitely generated and projective as a $k$-module, for
every $\lambda\in G$. Combining \prref{4.4} and \thref{3.4}, we find a $\GG$-graded
algebra $D(\ul{H})$ such that the categories $\YDT^{\ul{H}}_{\ul{H}}$ and
$\Tt^\GG_{D(\ul{H})}$, resp. $\YDZ^{\ul{H}}_{\ul{H}}$ and
$\Zz^\GG_{D(\ul{H})}$,
are isomorphic. $D(\ul{H})$ is called the Drinfeld double of 
$\ul{H}$, and can be described in two isomorphic ways: as a smash product or as
a Koppinen smash product. A straightforward computation based on our previous results
leads to these constructions.\\

\ul{Smash product.} 
$$D(\ul{H})=\bigoplus_{\lambda\in G}\bigoplus_{g\in G} H^*_{\lambda^{-1}}\# H_g.$$
We describe the multiplication on $D(\ul{H})$. Let $\xi\in H^*_{\lambda^{-1}}$,
$\xi'\in H^*_{\lambda'^{-1}}$, $h\in H_g$ and $h'\in H_{g'}$. Also assume that
$g'=g\gamma(\lambda')=\lambda'^{-1}g\lambda'=g^{\lambda'}$.\\
First recall the following notation. For $k,k',l\in H_{\lambda^{-1}}$, $k\leftactr \xi\rightactr k'\in
H^*_{\lambda^{-1}}$ is defined by
$$(k\leftactr \xi\rightactr k')(l)=\xi(k'lk).$$
Then
$$(\xi\# h)(\xi'\# h')=\xi(h_{(3, \lambda'^{-1})}\leftactr \xi'\rightactr S_{\lambda'^{-1}}(h_{(1, \lambda')}))
\# h_{(2,g')}h'.$$

\ul{Koppinen smash product.} 
$$D(\ul{H})=\bigoplus_{\lambda\in G}\bigoplus_{g\in G} \Hom(H_{\lambda^{-1}},H_g).$$
Take $f:\ H_{\lambda^{-1}}\to H_g$, $f':\ H_{\lambda'^{-1}}\to H_{g'}$, with $g'=g^{\lambda'}$.
Then $f\# f':\ H_{(\lambda\lambda')^{-1}}\to H_{g'}$ is defined by
\begin{eqnarray*}
&&\hspace*{-9mm}
(f\#f')(h)
=
f(h_{(2,\lambda^{-1})})_{(2,g')}\\
&&
f'\Bigl(S_{\lambda'^{-1}}\bigl(f(h_{(2,\lambda^{-1})})_{(1,\lambda')}\bigr)
h_{(1,\lambda'^{-1})}f(h_{(2,\lambda^{-1})})_{(3,(\lambda')^{-1})}\Bigr).
\end{eqnarray*}
\end{blanco}

\section{$\GG$-graded bialgebras}\selabel{5}
\begin{definition}\delabel{5.2}
Let $A=\oplus_{\lambda,g\in G}A_{\lambda,g}$ be a $\GG$-graded algebra. We call
$A$ a $\GG$-graded bialgebra if we have the following additional 
structure on $A$: for every $\lambda\in G$, $(G,(A_{\lambda,g})_{g\in G})$ is a
semi-Hopf group coalgebra, with structure maps
$$\Delta_{\lambda,g,g_1}:\ A_{\lambda,gg_1}\to A_{\lambda,g}\ot A_{\lambda,g_1}~~;~~
\varepsilon_\lambda:\ A_{\lambda,e}\to k,$$
such that the following
compatibility conditions hold:
\begin{equation}\eqlabel{5.2.1}
\Delta_{\lambda\lambda',g^{\lambda'},g_1^{\lambda'}}(aa')=
a_{(1,g)}a'_{(1,g^{\lambda'})}\ot a_{(2,g_1)}a'_{(2,g_1^{\lambda'})},
\end{equation}
for all $a\in A_{\lambda,gg_1}$, $a'\in A_{\lambda',g^{\lambda'}g_1^{\lambda'}}$;
\begin{equation}\eqlabel{5.2.2}
\varepsilon_{\lambda\lambda'}(aa')=\varepsilon_\lambda(a)\varepsilon_{\lambda'}(a'),
\end{equation}
for all $a\in A_{\lambda,e}$, $a'\in A_{\lambda',e}$;
\begin{equation}\eqlabel{5.2.3}
\Delta_{e,gg_1}(1_{gg_1})=1_g\ot 1_{g_1};
\end{equation}
\begin{equation}\eqlabel{5.2.4}
\varepsilon_e(1_e)=1.
\end{equation}
\end{definition}

\deref{5.2} has a monoidal justification, similar to the monoidal justification of the definition of
a bialgebra. Let $\Cc$ be the category with objects of the form $(V,M=\oplus_{v\in V} M_v)$,
with $V$ a crossed $G$-set, and
every $M_v$ a $k$-module. A morphism $(V,M)\to (V',M')$ in $\Cc$ is
a couple $(\eta,f)$, where $\eta:\ V\to V'$ is a morphism of crossed $G$-sets and $f:M\to M'$ is
a $k$-linear map such that $f(M_v)\subset M'_{\eta(v)}$.
Now $\Cc$ is a monoidal category. The tensor product is defined as follows
$$(V,M)\ot (V',M')=(V\times V',M\ot M'),$$
with $M\ot M'=\oplus_{(v,v')\in V\times V'} M_v\ot M_{v'}$. The unit object is
$(\{*\},k)$.\\
Now let $A$ be a $\GG$-graded algebra, and consider the forgetful functor
$$U:\  \Zz_A^{\GG}\to \Cc.$$

\begin{theorem}\thlabel{5.3}
Let $G$ be a group, and $A$ a $\GG$-graded algebra. We have a bijective correspondence between
\begin{itemize}
\item monoidal structures on $\Zz_A^{\GG}$ such that the forgetful functor $U$ is
strictly monoidal;
\item $\GG$-graded bialgebra structures on $A$.
\end{itemize}
\end{theorem}

\begin{proof}
Assume that we have a monoidal structure on $\Zz_A^{\GG}$ such that $U$ is strictly monoidal.
We first describe the structure maps $\varepsilon_\lambda$ and $\Delta_{\lambda,g,g_1}$.
Then the unit object is $k$, with a certain right $A$-module structure. From \deref{3.0}, we know that
this action is determined by maps
$$\varepsilon_\lambda:\ k\ot A_{\lambda,e}=A_{\lambda,e}\to k,$$
otherwise stated
\begin{equation}\eqlabel{5.3.0}
\varepsilon_\lambda(a)=1_k\cdot a,
\end{equation}
for $a\in A_{\lambda,e}$.
From \exref{3.1}, we know that
$$(G\times G, \bigoplus_{\lambda,g\in G} A_{\lambda,g})\in \Zz_A^{\GG}.$$
As an object in $\Cc$,
$$(G\times G,A)\ot (G\times G,A)= (G\times G\times G\times G, \bigoplus_{\lambda,\lambda',g,g'\in G}
A_{\lambda,g}\ot A_{\lambda',g'}).$$
The crossed $G$-set structure on $G\times G\times G\times G$ is the following:
$$\omega(\lambda,g,\lambda',g')=gg'~~;~~
(\lambda,g,\lambda',g')\lambda''=(\lambda\lambda'',g^{\lambda''},\lambda'\lambda'',g'^{\lambda''}).$$
Now we have a right $A$-module structure on $A\ot A$. According to \deref{3.0}, this is given
by multiplication maps
$$(A_{\lambda,g}\ot A_{\lambda',g'})\ot A_{\lambda'',g^{\lambda''}g'^{\lambda''}}\to
A_{\lambda\lambda'',g^{\lambda''}}\ot A_{\lambda'\lambda'',g'^{\lambda''}}.$$
Take $\lambda=\lambda'=e$, and replace $\lambda''$ by $\lambda$; this gives multiplication maps
$$\psi_{\lambda,g,g'}:\ (A_{e,g}\ot A_{e,g'})\ot A_{\lambda,g^{\lambda}g'^{\lambda}}\to
A_{\lambda,g^{\lambda}}\ot A_{\lambda,g'^{\lambda}}.$$
Now we define $\Delta_{\lambda,g,g'}:\ A_{\lambda,gg'}\to A_{\lambda,g}\ot A_{\lambda,g'}$
as follows:
\begin{equation}\eqlabel{5.3.2a}
\Delta_{\lambda,g,g'}(a)=\psi_{\lambda,g^{\lambda^{-1}},g'^{\lambda^{-1}}}
\bigl((1_{g^{\lambda^{-1}}}\ot 1_{g'^{\lambda^{-1}}})\ot a\bigr)
=(1_{g^{\lambda^{-1}}}\ot 1_{g'^{\lambda^{-1}}})a.
\end{equation}
For later use, observe that, for $a\in A_{\lambda,g^\lambda,g'^{\lambda}}$,
\begin{equation}\eqlabel{5.3.2}
\Delta_{\lambda,g^\lambda,g'^\lambda}(a)=(1_g\ot 1_{g'})a.
\end{equation}
We now have to show that the maps $\varepsilon_\lambda$ and $\Delta_{\lambda,g,g'}$
satisfy the conditions of \deref{5.2}. Before we do this, we show that the right $A$-action on
$M\ot N$ is completely determined by the maps $\Delta_{\lambda,g,g'}$, for all
$M,N\in \Zz_A^\GG$. We proceed as follows.\\
Let $(V,M)\in \Zz_A^{\GG}$, and fix elements $v\in V$ and $m\in M_v$. Recall from \blref{1.2}
that $G\times G$ is a crossed $G$-set, with structure maps
$$(\lambda,g)\lambda'=(\lambda\lambda',g^{\lambda'})~~;~~\beta(\lambda,g)=g.$$
In \exref{3.1}, we have seen that
$$Z_{\nu(v)}=\{(\lambda,\nu(v)^\lambda)~|~\lambda\in G\}$$
 is a crossed $G$-subset of $G\times G$ and that
 $$\bigl(Z_{\nu(v)}, A^{(\nu(v))}=\bigoplus_{\lambda\in G} A_{\lambda,\nu(v)^\lambda}\bigr)
 \in \Zz_A^\GG.$$
Now $\eta:\ Z_{\nu(v)}\to V$, $\eta(\lambda,\nu(v)^\lambda)=v\lambda$ is a morphism of
crossed $G$-sets. Indeed,
$$\eta((\lambda,\nu(v)^\lambda)\lambda')=\eta(\lambda\lambda',\nu(v)^{\lambda\lambda'})
= v\lambda\lambda'=\eta(\lambda,\nu(v)^\lambda)\lambda',$$
and
$$(\nu\circ \eta)(\lambda,\nu(v)^\lambda)=\nu(v\lambda)=\nu(v)^\lambda=\beta(\lambda,
\nu(v)^\lambda).$$
Now we define
$$f_m:\ A^{(\nu(v))}\to M,~~f_m(a)=ma.$$
It follows from \equref{3.1.1} that $M_v A_{\lambda,\nu(v)^\lambda}\subset M_{v\lambda}$,
so
$$f_m(A_{\lambda,\nu(v)^\lambda})\subset M_{v\lambda},$$
and $(\eta,f_m):\ (Z_{\nu(v)}, A^{(\nu(v))})\to (V,M)$ is a morphism in $\Mm_A^\GG$.\\
Now take $(V',N)\in \Mm_A^\GG$, fix $v'\in V'$ and $n\in N_{v'}$, and
repeat the above construction. We obtain a morphism
$(\eta',g_n):\ (Z_{\nu'(v')}, A^{(\nu'(v'))})\to (V',N)$ in $\Mm_A^\GG$.\\
From the functoriality of the tensor product, it follows that $(\eta\eta',f_m\ot g_n)$ 
is a morphism in $\Zz_A^\GG$, in particular, $f_m\ot g_n$ is right $A$-linear. Now take
$a\in A_{\lambda,\nu(v)^\lambda\nu'(v')^{\lambda}}$. Since $f_m(1_{\nu(v)})=m$
and $g_n(1_{\nu'(v')})=n$, we find
\begin{eqnarray}
(m\ot n)a&=& \bigl((f_m\ot g_n)(1_{\nu(v)}\ot 1_{\nu'(v')})\bigr)a\nonumber\\
&=& (f_m\ot g_n)\bigl((1_{\nu(v)}\ot 1_{\nu'(v')})a\bigr)\nonumber\\
&\equal{\equref{5.3.2}}&(f_m\ot g_n)(\Delta_{\lambda,\nu(v)^\lambda,\nu'(v')^\lambda}(a))\nonumber\\
&=&ma_{(1,\nu(v)^{\lambda})}\ot na_{(2,\nu'(v')^{\lambda})}.\eqlabel{5.3.3}
\end{eqnarray}
We are now ready to show that each $A_\lambda$ is a semi-Hopf group coalgebra. 
The (trivial) associativity constraint $a_{A,A,A}$
$$
\bigl((G\times G,A)\ot (G\times G,A)\bigr)\ot (G\times G,A)\to
(G\times G,A)\ot \bigl((G\times G,A)\ot (G\times G,A)\bigr)$$
is a morphism in $\Zz_A^\GG$; in particular $a_{A,A,A}$ is right $A$-linear. For all
$a\in A_{\lambda,gg'g''}$, we have that
\begin{eqnarray*}
&&\hspace*{-2cm}
\bigl(1_{g^{\lambda^{-1}}}\ot (1_{g'^{\lambda^{-1}}}\ot 1_{g''^{\lambda^{-1}}})\bigr)a
\equal{\equref{5.3.3}}
a_{(1,g)}\ot (1_{g'^{\lambda^{-1}}}\ot 1_{g''^{\lambda^{-1}}})a_{(2,g'g'')}\\
&=& a_{(1,g)}\ot \Delta_{\lambda,g'g''}(a_{(2,g'g'')})
\end{eqnarray*}
equals
\begin{eqnarray*}
&&\hspace*{-2cm}
a_{A,A,A}\bigl((1_{g^{\lambda^{-1}}}\ot 1_{g'^{\lambda^{-1}}})\ot 1_{g''^{\lambda^{-1}}}\bigr)a\\
&=&a_{A,A,A}\Bigl(\bigl((1_{g^{\lambda^{-1}}}\ot 1_{g'^{\lambda^{-1}}})\ot 1_{g''^{\lambda^{-1}}}\bigr)a\Bigr)\\
&\equal{\equref{5.3.3}}&
a_{A,A,A}\bigl( (1_{g^{\lambda^{-1}}}\ot 1_{g'^{\lambda^{-1}}})a_{(1,gg')}\ot a_{(2,g'')}\bigr)\\
&=& a_{A,A,A}(\Delta_{\lambda,g,g'}(a_{(1,gg')})\ot a_{(2,g'')}),
\end{eqnarray*}
which is precisely the required coassociativity condition. Now we prove the counit conditions.
The (trivial) left counit constraint $l_A:\ (\{*\},k)\ot (G\times G,A)\to (G\times G,A)$ is a morphism
in $\Zz_A^\GG$, hence $l_A$ is right $A$-linear. For $a\in A_{\lambda,g}$, we have
\begin{eqnarray*}
a&=&
l_A(1_k\ot 1_{g^\lambda})a=l_A((1_k\ot 1_{g^\lambda})a)\equal{\equref{5.3.3}}
l_A(1_k.a_{(1,e)}\ot 1_{g^\lambda} a_{(2,g)})\\
&=& l_A(\varepsilon_\lambda(a_{(1,e)})\ot a_{(2,g)})=\varepsilon_\lambda(a_{(1,e)}) a_{(2,g)},
\end{eqnarray*}
The right counit property is handled in a similar way. Now let $a\in A_{\lambda,gg_1}$ and
$a'\in A_{\lambda',g^{\lambda'}g_1^{\lambda'}}$. Then $aa'\in A_{\lambda\lambda',g^{\lambda'}g_1^{\lambda'}}$
and
\begin{eqnarray*}
&&\hspace*{-2cm}
\Delta_{\lambda\lambda',g^{\lambda'},g_1^{\lambda'}}(aa')=
(1_{g^{\lambda^{-1}}}\ot 1_{g_1^{\lambda^{-1}}})aa'\\
&=&
(a_{(1,g)}\ot a_{(2,g_1)})a'\equal{\equref{5.3.3}}
a_{(1,g)}a'_{(1,g^{\lambda'})}\ot a_{(2,g_1)}a'_{(2,g_1^{\lambda'})}.
\end{eqnarray*}
This proves that \equref{5.2.1} holds. Now take $a\in A_{\lambda,e}$ and $a'\in A_{\lambda',e}$.
Then
$$\varepsilon_{\lambda\lambda'}(aa')=1_k\cdot (aa')=(1_k\cdot a)\cdot a'=
\varepsilon_\lambda(a)\cdot a'=\varepsilon_\lambda(a)\varepsilon_{\lambda'}(a'),$$
proving \equref{5.2.2}. Finally
$$\Delta_{e,gg'}(1_{gg'})=(1_g\ot 1_{g'})1_{gg'}\equal{\equref{3.1.0}}1_g\ot 1_{g'},$$
and
$$\varepsilon_e(1_e)=1_k\cdot 1_e\equal{\equref{3.1.0}} 1_k.$$
Conversely, assume that $A$ is a $\GG$-graded bialgebra. Let $(V,M),~(V',M')\in \Zz_A^\GG$.
We have already seen that $V\times V'$ is again a crossed $G$-set. Now we define a right
$A$-module structure on $M\ot M'=\oplus_{(v,v')\in V\times V'} M_v\ot M'_{v'}$ using \equref{5.3.3},
which is designed in such a way that $(V\times V',M\ot M')\in \Zz_A^\GG$. Also $(\{*\},k)\in \Zz_A^\GG$,
using \equref{5.3.0}. Then straightforward computations show that this makes $\Zz_A^\GG$ into
a monoidal category such that the forgetful functor to $\Cc$ is strictly monoidal.\\
Let us show that the tensor on $\Zz_A^\GG$ is functorial. Consider morphisms
$$(\eta,\varphi):\ (V,M)\to (W,N)~~;~~(\eta',\varphi'):\ (V',M')\to (W',N')$$
in $\Zz_A^\GG$. The diagram \equref{3.2.5} takes the form
$$\xymatrix{
M_v\ot M'_{v'}\ot A_{\lambda,\nu(v)^\lambda\nu'(v')^{\lambda}}
\ar[d]_{\varphi_v\ot \varphi'_{v'}\ot id}
\ar[rr]^{}&& M_{v\lambda}\ot M'_{v'\lambda}
\ar[d]^{\varphi_{v\lambda}\ot \varphi'_{v'\lambda}}\\
N_{\eta(v)}\ot N'_{\eta'(v')}\ot A_{\lambda,\nu(v)^\lambda\nu'(v')^{\lambda}}
\ar[rr]^{}&& N_{\eta(v\lambda)}\ot N'_{\eta'(v'\lambda)}
}$$
and we have to show that it commutes. Since $(\eta,\varphi)$ and $(\eta',\varphi')$ are morphisms
in $\Zz_A^\GG$, we have, for $m\in M_v$, $a_1\in A_{\lambda,\nu(v\lambda)}$,
$m'\in M'_{v'}$ and $a_2\in A_{\lambda,\nu'(v'\lambda)}$ that
$$\varphi_{v\lambda}(ma_1)=\varphi_v(m)a_1~~{\rm and}~~
\varphi'_{v'\lambda}(m'a_2)=\varphi'_{v'}(m)a_2.$$
Now take $a\in A_{\lambda,\nu(v)^\lambda\nu'(v')^{\lambda}}$. Then we have
\begin{eqnarray*}
&&\hspace*{-15mm}
\varphi_{v\lambda}\bigl(ma_{(1,\nu(v)^\lambda)}\bigr)
\ot \varphi'_{v'\lambda}\bigl(m'a_{(2,\nu'(v')^\lambda)}\bigr)\\
&=& \varphi_v(m) a_{(1,\nu(v)^\lambda)}
\ot \varphi'_{v'}(m')a_{(2,\nu'(v')^\lambda)}
= \bigl((\varphi_v\ot  \varphi'_{v'})(m\ot m')\bigr)a,
\end{eqnarray*}
as needed.
\end{proof}

It would be nice to have a result similar to \thref{5.3}, with the category $\Zz_A^{\GG}$
replaced by $\Tt_A^{\GG}$. Unfortunately, we were only able to prove it in one direction.
First, we need to introduce the category $\Dd$, which can be viewed as the $\Tt$-version
of $\Cc$: it has the same objects as $\Cc$, and a morphism $(V,M)\to (W,N)$ is a couple
$(\eta, (\varphi_w)_{w\in W})$, with $\eta:\ W\to V$ a morphism of crossed $G$-sets,
and $\varphi_w:\ M_{\eta(w)}\to N_w$ $k$-linear, for all $w\in W$.

\begin{proposition}\prlabel{5.3bis}
Let $G$ be a group, and $A$ a $\GG$-graded bialgebra. Then we have a monoidal
structure on $\Tt_A^{\GG}$ such that the forgetful functor $\Tt_A^{\GG}\to \Dd$ is
monoidal.
\end{proposition}

\begin{proof}
As in the proof of \thref{5.3}, we define a right $A$-module structure on
$(V\times V',M\ot M')$ using \equref{5.3.3}, and on $(\{*\},k)$ using \equref{5.3.0}.
Let us show that the tensor product on $\Tt_A^{\GG}$ is functorial. Take
morphisms
$(\eta,(\varphi_w)_{w\in W}):\ (V,M)\to (W,N)$ and
$(\eta',(\varphi'_{w'})_{w'\in W'}):\ (V',M')\to (W',N')$
in $\Tt_A^{\GG}$. The diagram \equref{3.2.6} takes the form
$$\xymatrix{
M_{\eta(w)}\ot M'_{\eta'(w')}\ot A_{\lambda,\nu(\eta(w))^\lambda \nu'(\eta'(w'))^\lambda}
\ar[d]_{\varphi_w\ot \varphi'_{w'}\ot id}\ar[rr]^{}
&&M_{\eta(w\lambda)} \ot M'_{\eta'(w'\lambda)}
\ar[d]^{\varphi_{w\lambda}\ot \varphi'_{w'\lambda'}}
\\
N_w\ot N'_{w'}\ot A_{\lambda,\omega(w)^\lambda\omega'(w')^\lambda}
\ar[rr]^{}&& N_{w\lambda}\ot N'_{w'\lambda}}$$
and we have to show that it commutes. 
$(\eta,(\varphi_w)_{w\in W})$ and $(\eta',(\varphi'_{w'})_{w'\in W'})$ are morphisms
in $\Tt_A^{\GG}$, so, for all $m\in M_{\eta(w)}$, $a_1\in A_{\lambda,\nu(\eta(w))^\lambda}$,
$m'\in M'_{\eta'(w')}$ and $a_2\in A_{\lambda,\nu'(\eta'(w'))^\lambda}$, we have that
$$\varphi_{w\lambda}(ma_1)=\varphi_w(m)a_1~~{\rm and}~~
\varphi'_{w'\lambda}(m'a_2)=\varphi_{w'}(m')a_2.$$
For $a\in A_{\lambda,\nu(\eta(w))^\lambda \nu'(\eta'(w'))^\lambda}$, we now compute
\begin{eqnarray*}
&&\hspace*{-15mm}
(\varphi_{w\lambda}\ot \varphi'_{\omega'\lambda})((m\ot m')a)\\
&=& \varphi_{w\lambda}(ma_{(1,\nu(\eta(w))^\lambda)})
\ot \varphi'_{w'\lambda}(m'a_{(2,\nu'(\eta'(w'))^\lambda)})\\
&=&
\varphi_{w}(m) a_{(1,\omega(w)^\lambda)}
\ot \varphi'_{w'}(m') a_{(2,\omega'(w')^\lambda)}
= (\varphi_{w}(m)\ot \varphi'_{w'}(m')) a,
\end{eqnarray*}
as needed.
\end{proof}

Let $\ul{H}$ be a Hopf group coalgebra. The category of Yetter-Drinfeld modules
$\YDT_{\ul{H}}^{\ul{H}}$ is obtained from the category $\Tt_{\ul{H}}$ using the
center construction, see \cite[Sec. 4]{cl}, and therefore $\YDT_{\ul{H}}^{\ul{H}}$ is
a braided monoidal category. For detail on the centre construction, we refer to
\cite[XIII.4]{k}.
We first describe the monoidal structure.
Take $(V,M),~(V',N)\in \YDT_{\ul{H}}^{\ul{H}}$.
$$(V,M)\ot (V',N)=(V\times V', (M_v\ot N_{v'})_{(v,v')\in V\times V'}),$$
with the following structure. We have already seen that $V\times V'$ is a right crossed $G$-set,
with $\omega:\ V\times V'\to G$, $\omega(v,v')=\nu(v)\nu'(v')$ and $(v,v')g=(vg,v'g)$.\\
$M_v\ot N_{v'}$ is a right $H_{\nu(v)\nu'(v')}$-module, with
\begin{equation}\eqlabel{5.4.1}
(m\ot n)h=mh_{(1,\nu(v))}\ot nh_{(2,\nu'(v'))};
\end{equation}
The coaction maps $\rho_{(v,v'),g}:\ M_{vg}\ot N_{v'g}\to M_v\ot N_{v'}\ot H_g$ are given by
\begin{equation}\eqlabel{5.4.2}
\rho_{(v,v'),g}(m\ot n)=m_{[0,v]}\ot n_{[0,v']}\ot m_{[1,g]}n_{[1,g]}.
\end{equation}
$(\{*\},k)\in\YDT_{\ul{H}}^{\ul{H}}$; we already know that the singleton $\{*\}$ is a right crossed
$G$-set; furthermore $k$ is an $H_e$-module via $\varepsilon$, and the coaction maps
$\rho_{*,g}:\ k\to k\ot H_g$ are given by $ \rho_{*,g}(1_k)=1_k\ot 1_g$.\\
Now we describe the braiding. The braiding isomorphism
$$(V\times V', (M_v\ot N_{v'})_{(v,v')\in V\times V'})\to
(V'\times V, (N_{v'}\ot M_v)_{(v',v)\in V'\times V})$$
is given by the following data:
$$c_{V',V}:\ V'\times V \to V\times V',~~c_{V',V}(v',v)=(v,v'\nu(v)).$$
$$t_{M,N,v',v}:\ M_v\ot N_{v'\nu(v)}\to N_{v'}\ot M_v,~~
t_{M,N,v',v}(m\ot n)=n_{[0,v']}\ot mn_{[1,\nu(v)]}.$$
As we have mentioned, this monoidal structure can be deduced from the center
construction, but it can also be verified directly that this defines a monoidal structure 
on $\YDT_{\ul{H}}^{\ul{H}}$.\\
$\YDZ_{\ul{H}}^{\ul{H}}$ is also a braided monoidal category. The tensor product is
defined using (\ref{eq:5.4.1}-\ref{eq:5.4.2}). The braiding isomorphism
$$(V\times V', (M_v\ot N_{v'})_{(v,v')\in V\times V'})\to
(V'\times V, (N_{v'}\ot M_v)_{(v',v)\in V'\times V})$$
is given by the following data:
$$c^{-1}_{V',V}:\ V\times V'\to V'\times V,~c_{V',V}(v,v')=(v'\nu(v)^{-1},v);$$
$\tilde{t}_{M,N,v,v'}:\ M_v\ot N_{v'}\to N_{v'\nu(v)^{-1}}\ot M_v$ is given by
\begin{equation}\eqlabel{braiding}
\tilde{t}_{M,N,v,v'}(m\ot n)=n_{[0,v'\nu(v)^{-1}]}\ot mn_{[1,\nu(v)]}.
\end{equation}
We will also need the inverse of the braiding of $(c^{-1},\tilde{t})$ 
$$(V'\times V, (N_{v'}\ot M_v)_{(v',v)\in V'\times V})\to
(V\times V', (M_v\ot N_{v'})_{(v,v')\in V\times V'}).$$
This is described by the data
$$c_{V',V}:\ V'\times V\to V\times V,~~c_{V',V}(v',v)=(v,v'\nu(v));$$
$\tilde{q}_{N,M,v',v}:\ N_{v'}\ot M_v\to M_v\ot N_{v'\nu(v)}$
is given by the formula
\begin{equation}\eqlabel{inversebraiding}
\tilde{q}_{N,M,v',v}(n\ot m)=
m\ol{S}_{\nu(v)}(n_{[1,\nu(v)^{-1}]})\ot n_{[0,v'\nu(v)]}.
\end{equation}

If $(V,(M_v)_{v\in V})\in \YDZ_{\ul{H}}^{\ul{H}}$, then it is easy to see that 
$(V,M=\oplus_{v\in V}M_v)\in \Cc$: 
every $M_v$ is a $k$-module. Thus we have a forgetful functor $U':\ \YDZ_{\ul{H}}^{\ul{H}}\to \Cc$,
and it is clear that $U'$ is strictly monoidal.\\
Now assume that every $H_g$ is finitely generated and projective as a $k$-module.
Then we have an isomorphism of categories $Z$ (\thref{3.4}) and a forgetful functor
$U$ as in \thref{5.3} such that the diagram of functors
$$\xymatrix{
\YDZ_{\ul{H}}^{\ul{H}}\ar[rr]^{Z}\ar[d]^{U'}&&\Zz^\GG_{D(\ul{H})}\ar[dll]^U\\
\Cc}$$
commutes. It follows from all these observations that $\Zz^\GG_{D(\ul{H})}$ is
a monoidal category and that $U$ is strictly monoidal. Then it follows from 
\thref{5.3} that $D(\ul{H})$ is a $\GG$-graded bialgebra.\\
Our aim is now to construct the comultiplication and counit maps on $D(\ul{H})$.\\
We know that $(G\times G, D(\ul{H}))\in \Zz^\GG_{D(\ul{H})}$, see \exref{3.1}. 
From \thref{3.4} and \prref{4.4}, we know that $H(G\times G, D(\ul{H}))=(G\times G, D(\ul{H})_{(\lambda,g)
\in G\times G})\in \YDZ_{\ul{H}}^{\ul{H}}$. We compute the structure maps, using the proof
of \thref{5.3}.\\
First, every $D(\ul{H})_{(\lambda,g)}=H^*_{\lambda^{-1}}\# H_g$ is a right $H_g$-module
in the obvious way: $(\xi\# h)h'=\xi\#hh'$.\\
The coaction maps
$$\rho_{(\lambda,g),\lambda'}:\ H^*_{(\lambda\lambda')^{-1}}\# H_{g^{\lambda'}}
\to (H^*_{\lambda^{-1}}\# H_g)\ot H_{\lambda'}$$
are given by
\begin{equation}\eqlabel{coaction}
\rho_{(\lambda,g),\lambda'}(\xi\# h)=(\xi\# h)(\xi^{(\lambda')}\# 1_g)\ot h^{(\lambda')},
\end{equation}
where we use the notation introduced in \blref{2.9}: $\xi^{(\lambda)}\ot h^{(\lambda)}$ is a
finite dual basis of $H_\lambda$. 
$H_\lambda$ is a finitely projective algebra, hence $H_\lambda^*$ is a coalgebra, with comultiplication
\begin{equation}\eqlabel{5.4.3}
\Delta(\xi)=\xi_{(1)}\ot \xi_{(2)}=\lan \xi, h^{(\lambda)}\ol{h}^{(\lambda)}\ran
\xi^{(\lambda)}\ot \ol{\xi}^{(\lambda)},
\end{equation}
where $\ol{\xi}^{(\lambda)}\ot \ol{h}^{(\lambda)}= \xi^{(\lambda)}\ot h^{(\lambda)}$ is a second copy
of the dual basis of $H_\lambda$. Since $\YDZ_{\ul{H}}^{\ul{H}}$ is monoidal, we have that
$$(G\times G\times G\times G, (D(\ul{H})_{\lambda, g}\ot D(\ul{H})_{\lambda', g'})\in \YDZ_{\ul{H}}^{\ul{H}},$$
and we compute
\begin{eqnarray*}
&&\hspace*{-2cm}
\rho_{(\lambda,g,\lambda,g'),\lambda^{-1}}\bigl((\varepsilon\# 1_{g^{\lambda^{-1}}})
\ot (\varepsilon\# 1_{g'^{\lambda^{-1}}})\bigr)\\
&\equal{\equref{5.4.2}}&
(\xi^{(\lambda^{-1})}\# 1_g)\ot (\ol{\xi}^{(\lambda^{-1})}\# 1_{g'})\ot h^{(\lambda^{-1})}\ol{h}^{(\lambda^{-1})}.
\end{eqnarray*}
Now we apply \equref{5.3.2a} to compute $\Delta_{\lambda,g,g'}:\ D(\ul{H})_{\lambda,gg'}\to 
D(\ul{H})_{\lambda,g}\ot D(\ul{H})_{\lambda,g'}$: for $\xi\in H^*_{\lambda^{-1}}$ and
$h\in H_{gg'}$, we have
\begin{eqnarray*}
&&\hspace*{-2cm}
\Delta_{\lambda,g,g'}(\xi \# h)
\equal{\equref{5.3.2a}}
\bigl((\varepsilon\# 1_{g^{\lambda^{-1}}})\ot (\varepsilon\# 1_{g'^{\lambda^{-1}}})\bigr)(\xi\# h)\\
&=& (\xi^{(\lambda^{-1})}\# 1_g)\ot (\ol{\xi}^{(\lambda^{-1})}\# 1_{g'})
\lan \xi, h^{(\lambda^{-1})}\ol{h}^{(\lambda^{-1})}\ran h\\
&\equal{(\ref{eq:5.4.1},\ref{eq:5.4.3})}&
(\xi_{(1)}\# h_{(1,g)})\ot (\xi_{(2)}\# h_{(2,g')}).
\end{eqnarray*}
Now we compute the counit maps $\varepsilon_\lambda:\ H^*_{\lambda^{-1}}\# H_e\to k$.
$k$ is a right $H^*_{\lambda^{-1}}\# H_e$-module, and
$$\varepsilon_\lambda(\xi\# h)\equal{\equref{5.3.0}} 1_k\cdot (\xi\# h)=\xi(1_{\lambda^{-1}})
\varepsilon(h),$$
since $\rho_{*,\lambda^{-1}}(1_k)=1_k\ot 1_{\lambda^{-1}}$. We conclude our computations
as follows.

\begin{proposition}\prlabel{5.4}
Let $\ul{H}$ be a Hopf group coalgebra, and assume that every $H_g$ is finitely generated
and projective as a $k$-module. Then $D(\ul{H})$ is a $\GG$-graded bialgebra, with
structure maps
\begin{eqnarray*}
&&\hspace*{-2cm}
\Delta_{\lambda,g,g'}:\ D(\ul{H})_{\lambda,gg'}\to 
D(\ul{H})_{\lambda,g}\ot D(\ul{H})_{\lambda,g'},\\
&&
\Delta_{\lambda,g,g'}(\xi \# h)=(\xi_{(1)}\# h_{(1,g)})\ot (\xi_{(2)}\# h_{(2,g')});\\
&&\hspace*{-2cm}
\varepsilon_\lambda:\ H^*_{\lambda^{-1}}\# H_e\to k,~~\varepsilon_\lambda(\xi\# h)=
\xi(1_{\lambda^{-1}})\varepsilon(h).
\end{eqnarray*}
\end{proposition}

Recall from \blref{4.5} that $D(\ul{H})$ can also be written as a Koppinen smash product.
Then the comultiplication maps
$$\Delta_{\lambda,g,g'}:\ \Hom(H_{\lambda^{-1}},H_{gg'})\to
\Hom(H_{\lambda^{-1}},H_{g}) \ot \Hom(H_{\lambda^{-1}},H_{g'})$$
can be characterized as follows: $\Delta_{\lambda,g,g'}(f)=f_{(1)}\ot f_{(2)}$ if and only
if
$$\Delta_{g,g'}(f(h_1h_2))=f(h_1)_{(1,g)}\ot f(h_2)_{(2,g')},$$
for all $h_1,h_2\in H_{\lambda^{-1}}^*$. The counit maps are the following:
$$\varepsilon_\lambda:\ \Hom(H_{\lambda^{-1}},H_e)\to k,~~
\varepsilon_\lambda(f)=(\varepsilon\circ f)(1_{\lambda^{-1}}).$$

Let $A$ be $\GG$-graded bialgebra. From \thref{5.3} and \prref{5.3bis}, it follows
that a monoidal structure on $\Zz_A^\GG$ such that $U$ is strictly monoidal
induces a monoidal structure on $\Tt_A^\GG$. The tensor product of objects
coincides in both categories.\\
On $\Zz_{D(\ul{H})}^\GG$, we have the monoidal structure transported using
the category isomorphism with $\YDZ_{\ul{H}}^{\ul{H}}$. We have also a monoidal
structure arising from the $\GG$-graded bialgebra structure on $D(\ul{H})$,
using \thref{5.3}. These two monoidal structures coincide, actually this is the way
the $\GG$-graded bialgebra structure on $D(\ul{H})$ is constructed in the proof of
\prref{5.4}.\\
The monoidal structure on $\YDT_{\ul{H}}^{\ul{H}}$ can be transported to a
monoidal structure on $\Tt_{D(\ul{H}}^\GG$. It follows easily from our previous
constructions that this monoidal structure is induced from the monoidal structure
on $\Zz_{D(\ul{H}}^\GG$. Hence this monoidal structure coincide with the
monoidal structures arising from the $\GG$-graded bialgebra structure on $D(\ul{H})$,
using \prref{5.3bis}.
We summarize these observations as follows.

\begin{theorem}\thlabel{5.5a}
The $\GG$-graded bialgebra structure on $D(\ul{H})$ from \prref{5.4} defines a
monoidal algebra structure on $\Zz_{D(\ul{H})}^\GG$ and $\Tt_{D(\ul{H})}^\GG$
(\thref{5.3} and \prref{5.3bis}) that are such that the category isomorphisms
$\Zz_{D(\ul{H})}^\GG\cong \YDZ_{\ul{H}}^{\ul{H}}$ and 
$\Tt_{D(\ul{H})}^\GG\cong \YDT_{\ul{H}}^{\ul{H}}$ 
are isomorphisms of monoidal categories.
\end{theorem}

\section{$\GG$-graded Hopf algebras}\selabel{5b}
\begin{definition}\delabel{5.5}
Let $A=\oplus_{\lambda,g\in G} A_{\lambda,g}$ be a $\GG$-graded bialgebra.
We call $A$ a $\GG$-graded Hopf algebra if there exist maps
$$S_{\lambda,g},\ol{S}_{\lambda,g}:\ A_{\lambda,g^{-1}}\to A_{\lambda^{-1},g^{\lambda^{-1}}}$$
such that
\begin{eqnarray}
a_{(1,g)}S_{\lambda,g}(a_{(2,g^{-1})})&=&a_{(2,g)}\ol{S}_{\lambda,g}(a_{(1,g^{-1})})=
\varepsilon_\lambda(a)1_{g^{\lambda^{-1}}}
\eqlabel{5.5.1};\\
S_{\lambda,g}(a_{(1,g^{-1})})a_{(2,g)}&=&\ol{S}_{\lambda,g}(a_{(2,g^{-1})})a_{(1,g)}=
\varepsilon_\lambda(a)1_{g},\eqlabel{5.5.2}
\end{eqnarray}
for all $a\in A_{\lambda,e}$.
The $S_{\lambda,g}$ ($\ol{S}_{\lambda,g}$) are called the (twisted) antipode maps.
\end{definition}

\begin{proposition}\prlabel{5.6}
Let $\ul{H}$ be a Hopf group coalgebra, and assume that every $H_g$ is finitely generated
and projective as a $k$-module. Then $D(\ul{H})$ is a $\GG$-graded Hopf algebra, with
(twisted) antipode maps
$$S_{\lambda,g},\ol{S}_{\lambda,g}:\ H^*_{\lambda^{-1}}\# H_{g^{-1}}\to H^*_{\lambda}\# H_{g^{\lambda^{-1}}},$$
\begin{eqnarray*}
S_{\lambda,g}(\xi\# h)&=&(\varepsilon\# S_g(h))(\xi\circ \ol{S}_{\lambda^{-1}}
\# 1_{g^{\lambda^{-1}}});\\
\ol{S}_{\lambda,g}(\xi\# h)&=&(\varepsilon\# \ol{S}_g(h))(\xi\circ {S}_{\lambda^{-1}}
\# 1_{g^{\lambda^{-1}}}).
\end{eqnarray*}
\end{proposition}

\begin{proof}
For $\xi\in H^*_{\lambda^{-1}}$ and $h\in H_e$, we have
\begin{eqnarray*}
&&\hspace*{-2cm}
(\xi_{(1)}\# h_{(1,g)}) S_{\lambda,g}(\xi_{(2)}\# h_{(2,g^{-1})})\\
&=&
(\xi_{(1)}\# h_{(1,g)})(\varepsilon\# S_g(h_{(2,g^{-1})}))(\xi_{(2)}\circ \ol{S}_{\lambda^{-1}}\# 1_{g^{\lambda^{-1}}})\\
&=&(\xi_{(1)}\# h_{(1,g)}S_g(h_{(2,g^{-1})}))(\xi_{(2)}\circ \ol{S}_{\lambda^{-1}}\# 1_{g^{\lambda^{-1}}})\\
&=& \varepsilon(h)(\xi_{(1)}\# 1_g)(\xi_{(2)}\circ \ol{S}_{\lambda^{-1}}\# 1_{g^{\lambda^{-1}}})\\
&=& \varepsilon(h) \xi(1_{\lambda^{-1}}) (\varepsilon\# 1_{g^{\lambda^{-1}}})=
\varepsilon_\lambda(\xi\# h)(\varepsilon\# 1_{g^{\lambda^{-1}}}),
\end{eqnarray*}
where we used the following property, for all $h\in H_e$:
\begin{eqnarray*}
&&\hspace*{-2cm}
(\xi_{(1)}(\xi_{(2)}\circ S_\lambda^{-1}))(h)=
\xi_{(1)}(h_{(2,\lambda^{-1})})\xi_{(2)}(S_\lambda^{-1}(h_{(1,\lambda)}))\\
&=&\xi(h_{(2,\lambda^{-1})}S_\lambda^{-1}(h_{(1,\lambda)}))=
\xi(1_{\lambda^{-1}})\varepsilon(h).
\end{eqnarray*}
This proves one equality of \equref{5.5.1}; the proof of three other equalities is similar 
and is left to the reader. 
\end{proof}

\section{Braidings and quasitriangular $\GG$-graded Hopf algebras}\selabel{6}
\begin{definition}\delabel{6.1}
Let $A$ be a $\GG$-graded bialgebra. $A$ is called quasitriangular if it comes equipped with
the following additional structure: for all $g,g'\in G$, we have
\begin{eqnarray*}
R_{g,g'}&=& R^1_{g,g'}\ot R^2_{g,g'}\in A_{g^{-1},gg'g^{-1}}\ot A_{e,g};\\
Q_{g,g'}&=&Q^1_{g,g'}\ot Q^2_{g,g'}\in A_{g,g^{-1}g'g}\ot A_{e,g},
\end{eqnarray*}
such that the following conditions are fulfilled:
\begin{equation}
R_{g,g'}Q_{gg'g^{-1},g}=Q_{g',g}R_{g,g^{-1}g'g}=1_{g'}\ot 1_g;\eqlabel{6.1.1}
\end{equation}
\begin{equation}
\Delta_{g,g'^g,g''^g}(R^1_{g,g'g''})\ot R^2_{g,g'g''}=
R^1_{g,g'}\ot \tilde{R}^1_{g,g''}\ot R^2_{g,g'}\tilde{R}^2_{g,g''}\eqlabel{6.1.2}
\end{equation}
in $A_{g,g'^g}\ot A_{g,g''^{g}}\ot A_{e,g}$;
\begin{equation}
R^1_{gg',g''} \ot \Delta_{e,g,g'}(R^2_{gg',g''})=
R^1_{g',g''}\tilde{R}^1_{g,g'g''g'^{-1}}\ot  \tilde{R}^2_{g,g'g''g'^{-1}} \ot R^2_{g',g''}\eqlabel{6.1.3}
\end{equation}
in $A_{(gg')^{-1},g''^{(gg')^{-1}}}\ot A_{e,g}\ot A_{e,g'}$.\\
In addition, we have
for all $a\in A_{\lambda,g^\lambda g'^\lambda}$ that
\begin{equation}\eqlabel{6.1.4}
\tau(\Delta_{\lambda, g^\lambda, g'^\lambda})R_{g^\lambda, g'^\lambda}=
R_{g,g'}\Delta_{\lambda,g'^{g^{-1}\lambda},g^\lambda}(a).
\end{equation}
Here $\tau$ is the switch map.
\end{definition}

\deref{6.1} has a monoidal categorical justification.
Let $G$ be a group, and $A$ a $\GG$-graded bialgebra. We know that $\Zz_A^\GG$ is a
monoidal category, and that the forgetful functor $U:\ \Zz_A^\GG\to \Xx^G_G$ is
monoidal. Let $\Xx^{G{\rm inv}}_G$ be $\Xx^G_G$ with the inverse braiding $c^{-1}$.
Then we can look at braidings on $\Zz_A^\GG$ such that $U$ preserves the
braiding. Such a braiding is of the form $(c^{-1},\tilde{t})$, where $c$ is the braiding on $ \Xx^G_G$
as described in \blref{4.1}. In \prref{5.3bis}, we have seen that we have a monoidal
structure on $\Tt_A^\GG$ such that $V:\ \Tt_A^\GG\to \Xx^G_G$ is
monoidal, and we can consider braidings on $\Tt_A^\GG$ of the form
$(c,t)$, i.e. they are such that $V$ preserves the braiding.

\begin{theorem}\thlabel{6.2}
Let $G$ be a group, and $A$ a $\GG$-graded bialgebra. There is a bijective correspondence
between the following data:
\begin{itemize}
\item braidings on $\Zz_A^\GG$ of the form $(c^{-1},\tilde{t})$;
\item braidings on $\Tt_A^\GG$ of the form $(c,t)$;
\item quasitriangular structures on $A$ as defined in \deref{6.1}.
\end{itemize}
\end{theorem}

\begin{proof}
Given a braiding $(c,t)$ on $\Tt_A^\GG$, a braiding $(c^{-1},\tilde{t})$ on $\Zz_A^\GG$ is
given by the formula
\begin{equation}\eqlabel{relation}
\tilde{t}_{M,N,v,v'}=t_{M,N,v,v'\nu(v)^{-1}},
\end{equation}
and vice versa.\\
Next assume that we have a braiding $(c^{-1},\tilde{t})$ on $\Zz_A^\GG$.
For $\ul{M}=(V,M),~\ul{N}=(V',N)\in \Zz_A^\GG$, we have the braiding morphism
$$(c^{-1}_{V',V},\tilde{t}_{M,N,v,v'}):\ (V\times V',M\ot N)\to (V'\times V, N\ot M).$$
Then we have
\begin{equation}\eqlabel{6.2.1}
\tilde{t}_{M,N}(M_v\ot N_{v'})\subset N_{v'\nu(v)^{-1}}\ot M_{v}.
\end{equation}
Let $V=V'=G^2$, $M=N=A$. Then
$$c^{-1}_{G^2,G^2}(\lambda,g,\lambda',g')=(\lambda'g^{-1},gg'g^{-1},\lambda ,g),$$
and $\tilde{t}_{A,A}:\ A\ot A\to A\ot A$ satisfies
$$\tilde{t}_{A,A}(A_{\lambda,g}\ot A_{\lambda',g'})\subset A_{\lambda'g^{-1},gg'g^{-1}}\ot A_{\lambda,g}.$$
Now let
\begin{equation}\eqlabel{6.2.2}
R_{g,g'}=R^1_{g,g'}\ot R^2_{g,g'}=\tilde{t}_{A,A}(1_g\ot 1_{g'})\in A_{g^{-1},gg'g^{-1}}\ot A_{e,g}.
\end{equation}
We will show that the braiding $\tilde{t}$ is completely determined by the $R_{g,g'}$.
Take $m\in M_v$, $n\in N_{v'}$. We have seen in the proof of \thref{5.3} that we have
morphisms
$$(\eta, f_m):\ (Z_{\nu(v)}, A^{(\nu(v))})\to (V,M),~~
(\eta',g_n):\ (Z_{\nu'(v')}, A^{(\nu'(v'))})\to (V',N)$$
in $\Zz_A^\GG$. From the naturality of $(c^{-1},\tilde{t})$, we have the following commutative 
diagram
$$\xymatrix{
(Z_{\nu(v)}\times Z_{\nu'(v')}, A^{(\nu(v))}\ot A^{(\nu'(v'))})
\ar[d]_{(\eta,\eta',f_m\ot g_n)}\ar[r]^{(c^{-1},\tilde{t})}&
( Z_{\nu'(v')}\times Z_{\nu(v)},  A^{(\nu'(v'))}\ot A^{(\nu(v))})
\ar[d]^{(\eta',\eta, g_n\ot f_m)}\\
(V\times V',M\ot M')\ar[r]^{(c^{-1}_{V',V},\tilde{t}_{M,N})}& (V'\ot V, N\ot M)}$$
Observe that $(e,\nu(v))\in Z_{\nu(v)}$, $(e,\nu'(v'))\in Z_{\nu'(v')}$, $1_{\nu(v)}\in A^{(\nu(v))}$,
$1_{\nu'(v')}\in A^{(\nu'(v'))}$. From the commutativity of the diagram, it then follows that
\begin{eqnarray}
&&\hspace*{-2cm}
\tilde{t}_{M,N}(m\ot n)= (\tilde{t}_{M,N}\circ (f_m\ot g_n))(1_{\nu(v)}\ot 1_{\nu'(v')})\nonumber\\
&=& ((g_n\ot f_m)\circ \tilde{t}_{A^{(\nu(v))},A^{(\nu'(v'))}})(1_{\nu(v)}\ot 1_{\nu'(v')})\nonumber\\
&=& (g_n\ot f_m)(R_{\nu(v),\nu'(v')})\nonumber\\
&=& nR^1_{\nu(v),\nu'(v')}\ot m R^2_{\nu(v),\nu'(v')}.\eqlabel{6.2.3}
\end{eqnarray}
The inverse braiding can be described in a similar way: by assumption, $\tilde{t}_{M,N}$ is
invertible, and
\begin{equation}\eqlabel{6.2.4}
\tilde{t}^{-1}_{M,N}(N_{\nu'}\ot M_\nu) \subset M_{v}\ot N_{v'\nu(v)}.
\end{equation}
In fact the inclusions in \equref{6.2.1} and \equref{6.2.4} are equalities, since $\tilde{t}_{M,N}$
is bijective. In particular, we find that
$$\tilde{t}^{-1}_{A,A} (A_{\lambda',g'}\ot A_{\lambda,g})= 
A_{\lambda,g}\ot A_{\lambda'g,g'^g}.$$
Now let 
\begin{equation}\eqlabel{6.2.5}
\tilde{t}^{-1}_{A,A}(1_{g'}\ot 1_g)=Q_{g,g'}=Q^2_{g',g}\ot Q^1_{g',g}\in A_{e,g}\ot A_{g,g'^g}.
\end{equation}
The $Q_{g',g}$ describe the inverse braiding completely. Arguments similar to the ones above
show that, for $m\in M_v$ and $n\in N_{v'}$:
\begin{equation}\eqlabel{6.2.6}
\tilde{t}^{-1}_{M,N}(n\ot m)=mQ^2_{\nu'(v'),\nu(v)}\ot nQ^1_{\nu'(v'),\nu(v)}.
\end{equation}
Then we compute that
\begin{eqnarray*}
&&\hspace*{-2cm}
1_g\ot 1_{g'}=(\tilde{t}^{-1}_{A,A}\circ \tilde{t}_{A,A})(1_g\ot 1_{g'})
\equal{\equref{6.2.2}} \tilde{t}^{-1}_{A,A}(R^1_{g,g'}\ot R^2_{g,g'})\\
&\equal{\equref{6.2.6}}& R^2_{g,g'}Q^2_{gg'g^{-1},g}      \ot R^1_{g,g'}Q^2_{gg'g^{-1},g};\\
&&\hspace*{-2cm}
1_{g'}\ot 1_g= (\tilde{t}_{A,A}\circ \tilde{t}^{-1}_{A,A})(1_{g'}\ot 1_g)
\equal{\equref{6.2.5}} \tilde{t}_{A,A}(Q^2_{g',g}\ot Q^1_{g',g})\\
&\equal{\equref{6.2.3}}& Q^1_{g',g}R^1_{g,g^{-1}g'g}    \ot Q^2_{g',g}R^2_{g,g^{-1}g'g}.
\end{eqnarray*}
This shows that \equref{6.1.1} holds.\\
From the fact that $(c^{-1},\tilde{t})$ is a braiding, it follows that
\begin{eqnarray}
\tilde{t}_{A,A\ot A}&=&(A\ot \tilde{t}_{A,A})\circ (\tilde{t}_{A,A}\ot A);\eqlabel{6.2.7}\\
\tilde{t}_{A\ot A, A}&=& (\tilde{t}_{A,A}\ot A)\circ (A\ot \tilde{t}_{A,A}).\eqlabel{6.2.8}
\end{eqnarray}
Now we compute that
\begin{eqnarray*}
&&\hspace*{-2cm} 
\tilde{t}_{A,A\ot A}(1_g\ot 1_{g'}\ot 1_{g''})
\equal{\equref{6.2.3}}
(1_{g'}\ot 1_{g''})R^1_{g,g'g''} \ot 1_g R^2_{g,g'g''} \\
&\equal{\equref{5.3.2}}&
\Delta_{g,g'^g,g''^g}(R^1_{g,g'g''})\ot R^2_{g,g'g''};\\
&&\hspace*{-2cm} 
\bigl((A\ot \tilde{t}_{A,A})\circ (\tilde{t}_{A,A}\ot A)\bigr)(1_g\ot 1_{g'}\ot 1_{g''})\\
&\equal{\equref{5.3.3}}&(A\ot \tilde{t}_{A,A})(R^1_{g,g'}\ot R^2_{g,g'}\ot 1_{g''})\\
&\equal{\equref{6.2.3}}&
R^1_{g,g'}\ot 1_{g''}\tilde{R}^1_{g,g''}\ot R^2_{g,g'}\tilde{R}^2_{g,g''}\\
&=& R^1_{g,g'}\ot \tilde{R}^1_{g,g''}\ot R^2_{g,g'}\tilde{R}^2_{g,g''}.
\end{eqnarray*}
This shows that \equref{6.1.2} holds. \equref{6.1.3} can be proved in a similar way:
\begin{eqnarray*}
&&\hspace*{-2cm}
\tilde{t}_{A\ot A, A}(1_g\ot 1_{g'}\ot 1_{g''})
\equal{\equref{6.2.3}}
1_{g''} R^1_{gg',g''}\ot (1_g\ot 1_{g'})R^2_{gg',g''}\\
&\equal{\equref{5.3.2}}&
R^1_{gg',g''} \ot \Delta_{e,g,g'}(R^2_{gg',g''});\\
&&\hspace*{-2cm} 
\bigl((\tilde{t}_{A,A}\ot A)\circ (A\ot \tilde{t}_{A,A})\bigr)(1_g\ot 1_{g'}\ot 1_{g''})\\
&\equal{\equref{5.3.2}}&
(\tilde{t}_{A,A}\ot A)(1_g\ot R^1_{g',g''}\ot R^2_{g',g''})\\
&\equal{\equref{6.2.3}}&
R^1_{g',g''}\tilde{R}^1_{g,g'g''g'^{-1}}\ot 1_g \tilde{R}^2_{g,g'g''g'^{-1}} \ot R^2_{g',g''}\\
&=&R^1_{g',g''}\tilde{R}^1_{g,g'g''g'^{-1}}\ot  \tilde{R}^2_{g,g'g''g'^{-1}} \ot R^2_{g',g''}.
\end{eqnarray*}
Now take $a\in A_{\lambda,g^\lambda g'^\lambda}$. Since $\tilde{t}_{A,A}$ is right
$A$-linear, we have
$$\tilde{t}_{A,A}((1_g\ot 1_{g'})a)=\tilde{t}_{A,A}(1_g\ot 1_{g'})a.$$
Now
\begin{eqnarray*}
&&\hspace*{-2cm}
\tilde{t}_{A,A}((1_g\ot 1_{g'})a)\equal{\equref{5.3.2}}\tilde{t}_{A,A}(\Delta_{\lambda,g^\lambda, g'^\lambda}(a))\\
&=&\tilde{t}_{A,A}(a_{(1,g^{\lambda})}\ot a_{(2,g'^{\lambda})})
\equal{\equref{6.2.3}} a_{(2,g'^{\lambda})}R^1_{g^{\lambda},g'^{\lambda}}\ot
a_{(1,g^{\lambda})}R^2_{g^{\lambda},g'^{\lambda}};\\
&&\hspace*{-2cm}
\tilde{t}_{A,A}(1_g\ot 1_{g'})a\equal{\equref{6.2.2}}
(R^1_{g,g'}\ot R^2_{g,g'})a
\equal{\equref{5.3.3}} R^1_{g,g'}a_{(1,g'^{g^{-1}\lambda})}\ot R^2_{g,g'}a_{(2,g^{\lambda})}.
\end{eqnarray*}
\equref{6.1.4} follows, and we have shown that the $R_{g,g'}$ define a quasitriangular
structure on $A$.\\
Conversely, if $A$ is quasitriangular. Then we define $\tilde{t}_{M,N}$ using \equref{6.2.3}.
A lengthy but straightforward computation shows that $(c^{-1},\tilde{t})$ is a braiding on $\Zz_A^\GG$.
\end{proof}

Now let $\ul{H}$ be a Hopf group coalgebra, and assume that every $H_g$ is finitely
generated and projective as a $k$-module. Then the category $\YDZ_{\ul{H}}^{\ul{H}}$
is braided monoidal, and is isomorphic to $\Zz_{D(\ul{H})}^\GG$. We know from \prref{5.6}
that $D(\ul{H})$ is a $\GG$-graded Hopf algebra, and it follows from \thref{6.2} that we
have a quasitriangular structure on $D(\ul{H})$. The corresponding $R$-matrices can
be computed easily. The coaction map
$$\rho_{(g^{-1}, gg'g^{-1}), g}:\ H_e^*\# H_{g'}\to (H_g^*\# H_{gg'g^{-1}}) \ot H_g$$
can be computed using \equref{coaction}. In particular 
\begin{equation}\eqlabel{6.10}
\rho_{(g^{-1}, gg'g^{-1}), g}(\varepsilon\# 1_{g'})=(\xi^{(g)}\# 1_{gg'g^{-1}})\ot h^{(g)}.
\end{equation}
Then
\begin{eqnarray*}
R_{g,g'}&\equal{\equref{6.2.2}}&
\tilde{t}_{D(\ul{H}), D(\ul{H})}\bigl( (\varepsilon \# 1_g)\ot (\varepsilon \# 1_{g'})\bigr)\\
&\equal{(\ref{eq:braiding},\ref{eq:6.10})}&
(\xi^{(g)}\# 1_{gg'g^{-1}}) \ot (\varepsilon \# 1_g)h^{(g)}
=(\xi^{(g)}\# 1_{gg'g^{-1}}) \ot (\varepsilon \# h^{(g)}).
\end{eqnarray*}
In a similar way, we can compute the $Q$-matrices. Using \equref{coaction}, we
compute $\rho_{(g,g^{-1}g'g),g^{-1}}:\ H_e^*\# H_{g'}\to (H_{g^{-1}}^*\# H_{g^{-1}g'g}) 
\ot H_{g^{-1}}$:
$$\rho_{(g,g^{-1}g'g),g^{-1}}(\varepsilon \# 1_{g'})=
(\xi^{(g^{-1})}\# 1_{g^{-1}g'g})\ot h^{(g^{-1})}.$$
Then we find
\begin{eqnarray*}
&&\hspace*{-10mm}
Q_{g,g'}\equal{\equref{6.2.4}}
\tilde{q}_{D(\ul{H},\ul{H}}\bigl( (\varepsilon \# 1_{g'})\ot (\varepsilon \# 1_g)\bigr)
\equal{\equref{inversebraiding}}
(\varepsilon \# 1_g)\ol{S}_g(h^{(g^{-1})}) \ot (\xi^{(g^{-1})}\# 1_{g^{-1}g'g})\\
&=&(\varepsilon \# \ol{S}_g(h^{(g^{-1})})) \ot (\xi^{(g^{-1})}\# 1_{g^{-1}g'g}).
\end{eqnarray*}

We summarize our results.

\begin{theorem}\prlabel{6.10}
Let $\ul{H}$ be a Hopf group coalgebra, and assume that every $H_g$ is finitely generated
and projective as a $k$-module. Then $D(\ul{H})$ is a 
quasitriangular $\GG$-graded Hopf algebra, with $R$- and $S$-matrices 
$$R_{g,g'}=(\xi^{(g)}\# 1_{gg'g^{-1}}) \ot (\varepsilon \# h^{(g)})~~;~~
Q_{g,g'}=(\varepsilon \# \ol{S}_g(h^{(g^{-1})})) \ot (\xi^{(g^{-1})}\# 1_{g^{-1}g'g}).$$
The isomorphisms between the categories $\YDZ_{\ul{H}}^{\ul{H}}$ and
$\Zz^\GG_{D(\ul{H})}$ and between $\YDT_{\ul{H}}^{\ul{H}}$ and
$\Tt^\GG_{D(\ul{H})}$
(see \blref{4.5}) are isomorphisms of braided monoidal categories.
\end{theorem}

\section{Appendix: generalized Yetter-Drinfeld modules}
A generalization of Yetter-Drinfeld modules was proposed in \cite{cmz}, see
also \cite{cmz2}. First one has to introduce Yetter-Drinfeld data. There is a functor
from Yetter-Drinfeld data to Doi-Hopf data, and the corresponding categories of
Yetter-Drinfeld modules and Doi-Hopf modules are isomorphic. This construction
was carried out in the category of vector spaces, but can be generalized to
symmetric monoidal categories. Let us give the definition of Yetter-Drinfeld data
in $\Tt_k$.\\
First we discuss discrete Yetter-Drinfeld data, these are Yetter-Drinfeld data 
in ${\Sets}$. This is a four-tuple 
$(L,G,\Lambda,X)$, where $L$ and $G$ are groups, $\Lambda$ is a monoid,
$\psi:\ \Lambda\to L$ and $\gamma:\ \Lambda\to G$ are monoid maps,
and $X$ is a set with compatible left $L$-action and right $G$-action.\\
A crossed $(L,G,\Lambda,X)$-set is a right $\Lambda$-set $V$ together with a
map $\nu:\ V\to X$ such that $\nu(y\lambda)=\psi(\lambda)^{-1}\nu(y)\gamma(\lambda)$.\\
If $(L,G,\Lambda,X)$ is a discrete Yetter-Drinfeld datum, then we have a
discrete Doi-Hopf datum $(L\times G, \Lambda,X)$, with $(\psi,\gamma):\
\Lambda\to L\times G$, and $x(l,g)=l^{-1}xg$. An $(L\times G, \Lambda,X)$-set
is the same as a crossed $(L,G,\Lambda,X)$-set.\\
An example of a discrete Yetter-Drinfeld datum is $\GG=(G,G,G)$, as discussed in
the previous Sections.\\
A Yetter-Drinfeld datum in $\Tt_k$ is a fourtuple $(\ul{K},\ul{H},\ul{A},\ul{C})$, where
\begin{itemize}
\item $\ul{K}=(L,(K_l)_{l\in L})$ and $\ul{H}=(G,(H_g)_{g\in G})$ are Hopf group-coalgebras;
\item $\ul{A}=(X, (A_x)_{x\in X})$ is a $(\ul{K},\ul{H})$-bicomodule algebra;
\item $\ul{C}=(\Lambda, (C_\lambda)_{\lambda\in \Lambda})$ is a $(\ul{K},\ul{H})$-bimodule coalgebra.
\end{itemize}
Then $(L,G,\Lambda,X)$ is a discrete Yetter-Drinfeld datum; we have coaction maps
$$\rho_{l,x,g}:\ A_{lxg}\to K_l\ot A_x\ot H_g,~~
\rho_{l,x,g}(a)=a_{[-1,l]}\ot a_{[0,x]}\ot a_{[1,g]},$$
(Sweedler notation). $C_\lambda$ is a $(K_{\psi(\lambda)},H_{\gamma(\lambda)})$-bimodule,
for every $\lambda\in \Lambda$.\\
A Yetter-Drinfeld module is a couple $(V,(M_v)_{v\in V})$, where $V$ is
a crossed $(L,G,\Lambda,X)$-set, every $M_v$ is a right $A_{\nu(v)}$-module,
and $\ul{M}$ is a right $\ul{C}$-comodule, with structure maps
$\rho_{v,\lambda}:\ M_{v\lambda}\to M_v\ot C_\lambda$ such that the compatibility
relation
$$\rho_{v,\lambda}(ma)=
m_{[0,v]}a_{[0,\nu(v)]}\ot S_{\psi(\lambda)}(a_{[-1,\psi(\lambda)^{-1}]})m_{[1,\lambda]}
a_{[2,\gamma(\lambda)]}$$
holds for all $m\in M_{v\lambda}$ and $a\in A_{\nu(v\lambda)}=
A_{\psi(\lambda)^{-1}\nu(v)\gamma(\lambda)}$.\\
If $(\ul{K},\ul{H},\ul{A},\ul{C})$ is a Yetter-Drinfeld datum in $\Tt_k$, then
$(\ul{K}^{\rm op}\ot \ul{H},\ul{A},\ul{C})$ is a Doi-Hopf datum in $\Tt_k$:
$\ul{A}$ is a right $\ul{K}^{\rm op}\ot \ul{H}$-comodule algebra with coaction maps
$$\rho_{x,(l,g)}:\ A_{l^{-1}xg}\to A_x\ot K_l\ot H_g,~~
\rho_{x,(l,g)}(a)= a_{[0,x]}\ot S_l(a_{[-1,l^{-1}]})\ot a_{[1,g]}.$$
$\ul{C}$ is a right $\ul{K}^{\rm op}\ot \ul{H}$-module coalgebra, since 
every $C_\lambda$ is a right $K_{\psi(\lambda)}^{\rm op}\ot H_\gamma(\lambda)$-module.
Yetter-Drinfeld modules over $(\ul{K},\ul{H},\ul{A},\ul{C})$ then coincide with
Doi-Hopf modules over $(\ul{K}^{\rm op}\ot \ul{H},\ul{A},\ul{C})$. We can then consider
the categories $\YDT(\ul{K},\ul{H})_{\ul{A}}^{\ul{C}}$ and
$\YDZ(\ul{K},\ul{H})_{\ul{A}}^{\ul{C}}$ which are respectively isomorphic to the
categories $\Tt_k(\ul{K}^{\rm op}\ot \ul{H})_{\ul{A}}^{\ul{C}}$ and
$\Zz_k(\ul{K}^{\rm op}\ot \ul{H})_{\ul{A}}^{\ul{C}}$.
Now the duality results
from \seref{3} can be applied.

\begin{example}
$(\ul{H},\ul{H},\ul{H},\ul{H})$ is a Yetter-Drinfeld datum in $\Tt_k$, and the
corresponding Yetter-Drinfeld modules are the Yetter-Drinfeld modules that we
considered in \seref{4}.
\end{example}

\begin{example}\exlabel{YDsing}
Let $(L,G,\Lambda,X)$ be a discrete Yetter-Drinfeld datum. The crossed 
$(L,G,\Lambda,X)$-structures on a singleton $\{*\}$ are in bijective correspondence
with $X_0=\{x_0\in X~|~x_0\gamma(\lambda)=\psi(\lambda)x_0,~{\rm for~all~}\lambda\in \Lambda\}$.
The right $\Lambda$-action on $\{*\}$ is the trivial one, and $\nu(*)=x_0$.
In the case where $L=G=\Lambda=X$, $X_0$ is just the center of $G$.\\
Let $(\ul{K},\ul{H},\ul{A},\ul{C})$ a Yetter-Drinfeld datum in $\Tt_k$, and fix $x_0\in X_0$.
An $x_0$-Yetter-Drinfeld module
is an $(\ul{K},\ul{H},\ul{A},\ul{C})$-Yetter-Drinfeld module of the form
$(\{*\}, M)$, with $\nu(*)=x_0$. The full subcategory of $\YDT(\ul{K},\ul{H})_{\ul{A}}^{\ul{C}}$
consisting of $x_0$-Yetter-Drinfeld modules will be denoted by
$\YDT_{x_0}(\ul{K},\ul{H})^{\ul{C}}_{\ul{A}}$.
\end{example}

\begin{example}\exlabel{8.3}
We consider a particular instance of \exref{YDsing}. At the discrete level, take
$L=G=\Lambda=X$. The left and right $G$-action on $X=G$ are given by
multiplication. We fix $x_0\in X=G$, and define $\psi,~\gamma:\ \Lambda\to G$
by $\psi(g)=x_0gx_0^{-1}$ and $\gamma(g)=g$. It is then easy to see that
$x_0\in X$. We thus have a discrete Yetter-Drinfeld datum, which we will denote
by $(G,G,G, {}_{x_0}G)$.\\
Let $\ul{H}$ be a Hopf group coalgebra, with underlying group $G$; we
construct a Yetter-Drinfeld datum in $\Tt_k$ with underlying discrete Yetter-Drinfeld
datum $(G,G,G, {}_{x_0}G)$. Let $\ul{K}=\ul{H}$ and $\ul{A}=\ul{H}$, with $\ul{H}$-bicomodule
algebra structure induced by the comultiplication maps. Now we make $\ul{C}=\ul{H}$
into an $H$-module coalgebra. Every $H_\lambda$ is a right $H_\lambda$-module,
by multiplication. Consider a family of algebra maps
$\varphi=(\varphi_\lambda:\ H_{x_0\lambda x_0^{-1}}\rightarrow H_\lambda)_{\lambda\in G}$.
$\varphi_\lambda$ defines a left $H_{x_0\lambda x_0^{-1}}$-module structure on
$H_\lambda$ by restriction of scalars, and this makes $H_\lambda$ a
$(H_{x_0\lambda x_0^{-1}},H_\lambda)$-bimodule. This defines a left
$\ul{H}$-bimodule coalgebra structure on $\ul{H}$ if and only if
\begin{equation}\eqlabel{crossedsystem2}
\varepsilon_{\ul{H}}\varphi_e=\varepsilon_{\ul{H}}~~{\rm and}~~
\Delta_{\lambda, \lambda'}\circ \varphi_{\lambda\lambda'}=(\varphi_\lambda\ot \varphi_{\lambda'})\circ 
\Delta_{x_0\lambda x_0^{-1}, x_0\lambda' x_0^{-1}},
\end{equation}
for all $\lambda, \lambda'\in G$. The resulting $\ul{H}$-bimodule coalgebra will
be denoted ${}_{x_0,\varphi}\ul{H}$, and ${\YDT_{x_0,\varphi}}_{\ul{H}}^{\ul{H}}$
will be a shorter notation for the category 
$\YD_{x_0}(\ul{H},\ul{H})_{\ul{H}}^{{}_{x_0,\varphi}\ul{H}}$.
This definition of an $x_0\hbox{-}(\ul{H},\ul{H},\ul{H},{}_{x_0,\varphi}\ul{H})$-Yetter-Drinfeld
module agrees with 
the right version of 
$x_0$-Yetter-Drinfeld module over a $T$-coalgebra $\ul{H}$ as introduced by Zunino in 
\cite{zun}. Recall that a $T$-coalgebra is a Hopf group-coalgebra 
$\ul{H}=(G, (H_g)_{g\in G})$ together with a family of $k$-algebra isomorphisms
$\varphi=(\varphi^\sigma_\tau:\ H_\sigma\rightarrow H_{\tau\sigma\tau^{-1}})_{\sigma, \tau\in G}$ 
satisfying, among other, the conditions
$$\varepsilon_{\ul{H}}\varphi^e_\tau=\varepsilon_{\ul{H}}~~{\rm and}~~
\Delta_{\theta\sigma\theta^{-1}, \theta\tau\theta^{-1}}\circ 
\varphi^{\sigma\tau}_\theta =(\varphi^\sigma_\theta\otimes \varphi^\tau_\theta)\circ \Delta_{\sigma, \tau}.$$
for all $\sigma, \tau, \theta\in G$. Fix $x_0\in G$, and define  
$\varphi_\lambda=\varphi^{x_0\lambda x_0^{-1}}_{x_0^{-1}}$, 
for any $\lambda\in G$. Then 
$\varepsilon_{\ul{H}}\varphi_e=\varepsilon_{\ul{H}}\varphi^e_{x_0^{-1}}=\varepsilon_{\ul{H}}$ and the family 
$\ul{\varphi}:=(\varphi_\lambda=\varphi^{x_0\lambda x_0^{-1}}_{x_0}: H_{x_0\lambda x_0^{-1}}\rightarrow H_\lambda)_{\lambda\in G}$ satisfies 
\equref{crossedsystem2}. To see this take  $\theta=x_0^{-1}$, 
$\sigma=x_0\lambda x_0^{-1}$ and $\tau=x_0\lambda' x_0^{-1}$, where $\lambda, \lambda'\in G$
in the above equality. In this situation ${\YDT_{x_0,\varphi}}_{\ul{H}}^{\ul{H}}$ is precisely the category of 
right $x_0$-Yetter-Drinfeld modules over a $T$-coalgebra, in the spirit of \cite{zun}. 
\end{example}

\begin{example} \exlabel{8.4}
We present a variation of \exref{8.3}. At the discrete level, let $\psi$ be the identity on $G$,
and let $\gamma$ be conjugation by a fixed $x_0\in G$: $\gamma(g)=x_0^{-1}gx_0$.
An $\ul{H}$-bimodule structure on $\ul{H}$ can be obtained using a family of
algebra maps $\varphi'=(\varphi'_\lambda:\ H_{x_0^{-1}\lambda x_0}\rightarrow H_\lambda)_{\lambda\in G}$. The $(H_\lambda,H_{x_0^{-1}\lambda x_0})$-bimodule structure on $H_\lambda$
is obtained via restriction of scalars, using the identity on the left and $\varphi_\lambda$
on the right hand side. This defines an $\ul{H}$-bimodule coalgebra structure on $\ul{H}$
if and only if  
\begin{equation}\eqlabel{crossedsystem1}
\varepsilon_{\ul{H}}\varphi'_e=\varepsilon_{\ul{H}}~~{\rm and}~~
\Delta_{\lambda, \lambda'}\circ \varphi'_{\lambda\lambda'}=(\varphi'_\lambda\otimes \varphi'_{\lambda'})\circ 
\Delta_{x_0^{-1}\lambda x_0, x_0^{-1}\lambda'x_0},
\end{equation}
for all $\lambda, \lambda'\in G$. The resulting $\ul{H}$-bimodule coalgebra is denoted
by $\ul{H}_{x_0,\varphi'}$, and we use the shorter notation 
${\YDT_{\ul{H}}^{\ul{H}}}_{x_0,\varphi}$ for
the category $\YD_{x_0}(\ul{H},\ul{H})_{\ul{H}}^{\ul{H}_{x_0,\varphi}}$.\\
Particular examples can be deduced from $T$-coalgebras. More precisely, let $\ul{H}$ be a 
$T$-coalgebra and 
$\varphi=(\varphi^\sigma_\tau: H_\sigma\rightarrow H_{\tau\sigma\tau^{-1}})_{\sigma, \tau\in G}$  
the conjugation of $\ul{H}$. We have a family of algebra morphisms  
$\ul{\varphi'}=(\varphi'_\lambda=\varphi^{x_0^{-1}\lambda x_0}_{x_0}: H_{x_0^{-1}\lambda x_0}\rightarrow H_\lambda)_{\lambda\in G}$. 
A simple inspection shows that $\ul{\varphi'}$ satisfies 
\equref{crossedsystem1}. Thus it is possible to define the notion of 
$x_0$-Yetter-Drinfeld module in a way that is different from the one in \cite{zun}.
\end{example}


\begin{thebibliography}{99}
\bibitem{Borceux}
F. Borceux, ``Handbook of categorical algebra I: basic category theory",
{\sl Encyclopedia Math. Appl.} {\bf 50}, Cambridge University Press, Cambridge, 1994.

\bibitem{cl}
S. Caenepeel, M. De Lombaerde, {\sl A categorical approach to Turaev's Hopf group-coalgebras}, 
Comm. Algebra {\bf 34} (2006), 2631--2657.

\bibitem{cmz}
S. Caenepeel, G. Militaru and S. Zhu, {\sl Crossed modules and Doi-Hopf modules}, 
Israel J. Math. {\bf 100} (1997), 221--247.

\bibitem{cmz2}
S. Caenepeel, G. Militaru and S. Zhu, ``Frobenius and Seperable Functors for Generalized Module categories and Nonlinear Equations", 
LNM 1787, Springer, 2002. 

\bibitem{FY}
P. Freyd, D. Yetter, {\sl Braided compact closed categories with applications to
low-dimensional topology}, Adv. Math. {\bf 77} (1989), 156--182.

\bibitem{ph}
D. Hobst and B. Pareigis, {\sl Double quantum groups}, J. Algebra {\bf 242}(2) (2001), 
460--494.

\bibitem{k}
C. Kassel, ``Quantum Groups", {\sl Graduate Texts in Mathematics}
{\bf 155}, Springer Verlag, Berlin, 1995.

\bibitem{nvo}
C. N\v ast\v asescu, F. Van Oystaeyen, ``Graded ring theory", North Holland, Amsterdam, 1982.

\bibitem{NastasescuRV90}
C. N\v ast\v asescu, \c S. Raianu and F. Van Oystaeyen, {\sl Modules
graded by $G$-sets}, {Math. Z.} {\bf 203} (1990), 605--627.

\bibitem{sw}
M. E. Sweedler, ``Hopf algebras'', Benjamin, New York, 1969.

\bibitem{turaev}
V. G. Turaev, {\sl Homotopy field theory in dimension $3$ and 
crossed group-categories}, preprint arXiv: math. GT/0005291. 

\bibitem{turaev2}
V. G. Turaev, {\sl Crossed group-categories}, Arab. J. Sci. Engineering {\bf 33} (2008),
483--503.
 
\bibitem{vire}
A. Virelizier, {\sl Hopf group-coalgebras},  J. Pure Appl. Algebra {\bf 171}(2002), 
75--122. 
\bibitem{wang}
S. H. Wang, {\sl Group entwining structures and group coalgebra extensions}, 
Comm. Algebra {\bf 32} (2004), 3437--3457.
\bibitem{wang2}
S. H. Wang, {\sl Group twisted smash products and Doi-Hopf modules for $T$-coalgebras}, 
Comm. Algebra {\bf 32} (2004), 3417--3436. 
\bibitem{zun}
M. Zunino, {\sl Yetter-Drinfeld modules for crossed structures},  
J. Pure Appl. Algebra {\bf 193} (2004), 313--343.


\end{thebibliography}
\end{document}